\documentclass[11pt,letterpaper]{amsart}

\usepackage[a4paper,left=3cm,right=3cm,top=3cm,bottom=2.5cm]{geometry}
\usepackage{comment}
\usepackage{xcolor}
\usepackage{amsmath}
\usepackage{mathtools}
\usepackage{booktabs}
\usepackage[all]{xy}
\usepackage[utf8]{inputenc}
\usepackage{varioref}
\usepackage{amsfonts}
\usepackage{amsmath}
\usepackage{amssymb}
\usepackage{mathabx}
\usepackage{bbm}
\usepackage{esint}
\usepackage{graphicx}
\usepackage{tikz}
\usepackage{empheq}
\usepackage[shortlabels]{enumitem}
\usepackage{tikz-cd}
\usepackage{todonotes}
\usepackage{bigints}
\usepackage{amsmath}
\usepackage{amsfonts}
\usepackage{amsthm}
\usepackage{lipsum}
\usepackage{amssymb}
\usepackage[utf8]{inputenc}
\usepackage[english]{babel}
\usepackage{pst-node}
\usepackage{tikz-cd} 
\usepackage{hyperref} 
\usepackage{mathtools}
\usepackage{xcolor}
\usepackage{centernot}
\usepackage{fancyhdr}
\usepackage[T1]{fontenc}
\usepackage{mathabx,graphicx}
\usepackage[english]{babel}
\usetikzlibrary{matrix,arrows,decorations.pathmorphing}
\usepackage{mathrsfs}

\newcommand{\sH}{{\mathcal H}}

\newcommand{\sL}{{\mathcal L}}

\newcommand{\sV}{{\mathcal V}}

\newcommand{\C}{{\mathbb C}}

\renewcommand{\O}{{\rm O}}

\renewcommand{\to}[1][]{\xrightarrow{\ #1\ }}

\makeatletter
\newcommand*{\da@rightarrow}{\mathchar"0\hexnumber@\symAMSa 4B }
\newcommand*{\da@leftarrow}{\mathchar"0\hexnumber@\symAMSa 4C }
\newcommand*{\xdashrightarrow}[2][]{%
  \mathrel{%
    \mathpalette{\da@xarrow{#1}{#2}{}\da@rightarrow{\,}{}}{}%
  }%
}
\newcommand{\xdashleftarrow}[2][]{%
  \mathrel{%
    \mathpalette{\da@xarrow{#1}{#2}\da@leftarrow{}{}{\,}}{}%
  }%
}
\newcommand*{\da@xarrow}[7]{%
  \sbox0{$\ifx#7\scriptstyle\scriptscriptstyle\else\scriptstyle\fi#5#1#6\m@th$}%
  \sbox2{$\ifx#7\scriptstyle\scriptscriptstyle\else\scriptstyle\fi#5#2#6\m@th$}%
  \sbox4{$#7\dabar@\m@th$}%
  \dimen@=\wd0 %
  \ifdim\wd2 >\dimen@
    \dimen@=\wd2 %
  \fi
  \count@=2 %
  \def\da@bars{\dabar@\dabar@}%
  \@whiledim\count@\wd4<\dimen@\do{%
    \advance\count@\@ne
    \expandafter\def\expandafter\da@bars\expandafter{%
      \da@bars
      \dabar@ 
    }%
  }%
  \mathrel{#3}%
  \mathrel{%
    \mathop{\da@bars}\limits
    \ifx\\#1\\%
    \else
      _{\copy0}%
    \fi
    \ifx\\#2\\%
    \else
      ^{\copy2}%
    \fi
  }%
  \mathrel{#4}%
}
\makeatother

\makeatletter
\newsavebox\myboxA
\newsavebox\myboxB
\newlength\mylenA

\newcommand*\xtilde[2][0.8]{%
    \sbox{\myboxA}{$\m@th#2$}%
    \setbox\myboxB\null
    \ht\myboxB=\ht\myboxA%
    \dp\myboxB=\dp\myboxA%
    \wd\myboxB=#1\wd\myboxA
    \sbox\myboxB{$\m@th\widetilde{\copy\myboxB}$}
    \setlength\mylenA{\the\wd\myboxA}
    \addtolength\mylenA{-\the\wd\myboxB}%
    \ifdim\wd\myboxB<\wd\myboxA%
       \rlap{\hskip 0.5\mylenA\usebox\myboxB}{\usebox\myboxA}%
    \else
        \hskip -0.5\mylenA\rlap{\usebox\myboxA}{\hskip 0.5\mylenA\usebox\myboxB}%
    \fi}

\newbox\usefulbox

\def\getslant #1{\strip@pt\fontdimen1 #1}

\def\xxtilde #1{\mathchoice
 {{\setbox\usefulbox=\hbox{$\m@th\displaystyle #1$}%
    \dimen@ \getslant\the\textfont\symletters \ht\usefulbox
    \divide\dimen@ \tw@ 
    \kern\dimen@ 
    \xtilde{\kern-\dimen@ \box\usefulbox\kern\dimen@ }\kern-\dimen@ }}
 {{\setbox\usefulbox=\hbox{$\m@th\textstyle #1$}%
    \dimen@ \getslant\the\textfont\symletters \ht\usefulbox
    \divide\dimen@ \tw@ 
    \kern\dimen@ 
    \xtilde{\kern-\dimen@ \box\usefulbox\kern\dimen@ }\kern-\dimen@ }}
 {{\setbox\usefulbox=\hbox{$\m@th\scriptstyle #1$}%
    \dimen@ \getslant\the\scriptfont\symletters \ht\usefulbox
    \divide\dimen@ \tw@ 
    \kern\dimen@ 
    \xtilde{\kern-\dimen@ \box\usefulbox\kern\dimen@ }\kern-\dimen@ }}
 {{\setbox\usefulbox=\hbox{$\m@th\scriptscriptstyle #1$}%
    \dimen@ \getslant\the\scriptscriptfont\symletters \ht\usefulbox
    \divide\dimen@ \tw@ 
    \kern\dimen@ 
    \xtilde{\kern-\dimen@ \box\usefulbox\kern\dimen@ }\kern-\dimen@ }}%
 {}}

\newcommand*\xoverline[2][0.75]{%
    \sbox{\myboxA}{$\m@th#2$}%
    \setbox\myboxB\null
    \ht\myboxB=\ht\myboxA%
    \dp\myboxB=\dp\myboxA%
    \wd\myboxB=#1\wd\myboxA
    \sbox\myboxB{$\m@th\overline{\copy\myboxB}$}
    \setlength\mylenA{\the\wd\myboxA}
    \addtolength\mylenA{-\the\wd\myboxB}%
    \ifdim\wd\myboxB<\wd\myboxA%
       \rlap{\hskip 0.5\mylenA\usebox\myboxB}{\usebox\myboxA}%
    \else
        \hskip -0.5\mylenA\rlap{\usebox\myboxA}{\hskip 0.5\mylenA\usebox\myboxB}%
    \fi}

\def\xxoverline #1{\mathchoice
 {{\setbox\usefulbox=\hbox{$\m@th\displaystyle #1$}%
    \dimen@ \getslant\the\textfont\symletters \ht\usefulbox
    \divide\dimen@ \tw@ 
    \kern\dimen@ 
    \overline{\kern-\dimen@ \box\usefulbox\kern\dimen@ }\kern-\dimen@ }}
 {{\setbox\usefulbox=\hbox{$\m@th\textstyle #1$}%
    \dimen@ \getslant\the\textfont\symletters \ht\usefulbox
    \divide\dimen@ \tw@ 
    \kern\dimen@ 
    \xoverline{\kern-\dimen@ \box\usefulbox\kern\dimen@ }\kern-\dimen@ }}
 {{\setbox\usefulbox=\hbox{$\m@th\scriptstyle #1$}%
    \dimen@ \getslant\the\scriptfont\symletters \ht\usefulbox
    \divide\dimen@ \tw@ 
    \kern\dimen@ 
    \xoverline{\kern-\dimen@ \box\usefulbox\kern\dimen@ }\kern-\dimen@ }}
 {{\setbox\usefulbox=\hbox{$\m@th\scriptscriptstyle #1$}%
    \dimen@ \getslant\the\scriptscriptfont\symletters \ht\usefulbox
    \divide\dimen@ \tw@ 
    \kern\dimen@ 
    \xoverline{\kern-\dimen@ \box\usefulbox\kern\dimen@ }\kern-\dimen@ }}%
 {}}
\makeatother

\makeatletter
\newcommand{\mylabel}[2]{#2\def\@currentlabel{#2}\label{#1}}
\makeatother

\makeatletter
\newcommand{\Mac}{}
\DeclareRobustCommand{\Mac}{%
  M%
  \raisebox{\dimexpr\fontcharht\font`M-\height}{%
    \check@mathfonts\fontsize{\sf@size}{0}\selectfont
    c%
  }%
}
\makeatother

\newtheoremstyle{citing}
  {}
  {}
  {\itshape}
  {}
  {\bfseries}
  {\textbf{.}}
  {.5em}
  {\thmnote{#3}}

\theoremstyle{plain}
\newtheorem{theorem}{Theorem}[section]

\newtheorem{lemma}[theorem]{Lemma}
\newtheorem{corollary}[theorem]{Corollary}

\newtheorem{bigthm}{Theorem}
\newtheorem{bigcor}[bigthm]{Corollary}

\theoremstyle{remark}

\theoremstyle{definition}

\newtheorem{definition}[theorem]{Definition}

\numberwithin{equation}{section}

\theoremstyle{remark}
\newtheorem{remark}[theorem]{Remark}

{\theoremstyle{citing}
}

{\theoremstyle{definition}
}

\title[On the Rigidity of Hamiltonians which are Zoll Near a Minimum]{On the Rigidity of Hamiltonians which are Zoll Near a Minimum, with an Application to Magnetic Systems and Almost-Kähler Manifolds}

\author{Gabriele Benedetti, Johanna Bimmermann, Samanyu Sanjay}
\address{International School for Advanced Studies (SISSA), Via Bonomea 265, 34136, Trieste, Italy}
\email{gabriele.benedetti@sissa.it}

\address{Mathematical Institute, University of Oxford, Andrew Wiles Building, Woodstock Road, OX2 6GG, U.K.}
\email{johanna.bimmermann@maths.ox.ac.uk}

\address{RWTH Aachen, Lehrstuhl für Geometrie und Analysis, Pontdriesch 10-12,52062 Aachen}
\email{sajay@mathga.rwth-aachen.de}

\let\origmaketitle\maketitle
\def\maketitle{
  \begingroup
  \def\uppercasenonmath##1{} 
  \let\MakeUppercase\relax 
  \origmaketitle
  \endgroup
}

\makeatletter
\newcommand{\myabstract}[1]{\gdef\@myabstract{#1}}
\gdef\@myabstract{}
\let\oldmaketitle\maketitle
\renewcommand{\maketitle}{%
  \oldmaketitle
  \ifx\@myabstract\@empty\else
    \begin{center}
      \begin{minipage}{0.85\textwidth}
        \small\textbf{Abstract. }\@myabstract
      \end{minipage}
    \end{center}
    \vspace{1em}
  \fi
}
\makeatother

\begin{document}
\thispagestyle{empty}
\noindent\begin{abstract}
We study Hamiltonian systems near a compact symplectic Morse–Bott minimum. Our first result shows that if the flow is Zoll (that is, it induces a free circle action) along a sequence of energy levels converging to the minimum, then the Hessian of the Hamiltonian in the symplectic normal directions must be compatible with the restriction of the symplectic structure to the normal bundle (that is, its representing endomorphism is a complex structure of the symplectic normal bundle). For our second result, we specialize to magnetic systems on closed manifolds with symplectic magnetic form. In this setting, if the system is Zoll along a sequence of energy levels converging to the minimum, then the metric is compatible with the magnetic form and therefore defines an almost Kähler structure. We show that a natural curvature quantity, consisting of the holomorphic sectional curvature corrected by a term measuring the non-integrability of the almost complex structure, must be constant. In particular, we obtain a dynamical characterization of complex space forms among Kähler manifolds. Together, these results establish strong rigidity of systems which are Zoll at energies close to a Morse--Bott minimum, in the symplectic and in the magnetic settings.
\end{abstract}
\maketitle
\setlength{\parindent}{1em}
\setcounter{tocdepth}{1}

\tableofcontents

\section{Introduction}

The existence of periodic orbits near a symplectic Morse--Bott minimum has a long history. In the case where $Q$ is a point, the Weinstein--Moser theorem guarantees the existence of $k$ periodic orbits on every sufficiently low energy level \cite{Weinstein_Moser_Mo,Weinstein_Moser_W}. For a general symplectic minimum $Q$, sharper multiplicity results were later obtained, under global resonance assumptions on the eigenvalues of the linearized dynamics, by Ginzburg \cite{Gin87}, Ginzburg--Kerman \cite{GK}, and Kerman \cite{Kerman}, building on a theorem of Bottkol \cite{Bottkol}. In full generality, the existence of at least one periodic orbit on low energy levels was established using Floer-theoretic methods by Ginzburg--G\"urel \cite{Ginzburg_gurel} and Usher \cite{usher}; see also \cite{GK2,CGK,Schlenk,FF}. Motivated by this circle of ideas, in the present paper we study the emergence of \textit{global} periodic behavior at low energy levels near a symplectic Morse--Bott minimum, focusing on Hamiltonians that are Zoll along a sequence of energies converging to the minimum and proving strong rigidity results in terms of normal forms.

\subsection{The Symplectic Setting}\label{s:setting}
Let $(M,\omega)$ be a smooth symplectic manifold and let $H\colon M\to{} [0,\infty)$ be a smooth Hamiltonian function that reaches a Morse-Bott non-degenerate minimum at $0$ such that the submanifold $Q:= H^{-1}(0)$ is a non-empty, embedded, connected, closed submanifold of $M$. The Morse-Bott non-degeneracy of $H$ at zero implies that the Hessian of the Hamiltonian function $H$ at $Q$ yields a positive-definite inner product $\gamma$ on a normal bundle $\pi\colon E\to Q$ to $Q$ inside $M$. On $E$, $\gamma$ has the following coordinate expression:
\begin{equation}
\gamma_q(u,v):=\frac{\partial^2 H}{\partial u\partial v}(q),\qquad \forall\,q\in Q,\ \forall\, u,v\in E_q.
\end{equation}
We will be interested in the periodicity properties of the Hamiltonian flow $\Phi_H$ of $H$ for small energies $\tfrac12\varepsilon^2$, when $Q$ is a symplectic submanifold of $M$. 

The Hamiltonian flow $\Phi_H$ is obtained by integrating the Hamiltonian vector field $X_H$ on $M$, the vector field $X_H$ is defined to be the unique solution to the following differential equation:
\begin{equation}
\omega(X_H,\cdot)=-dH.
\end{equation}
$\Phi_H$ preserves the Hamiltonian function $H$, also known as the energy, and is therefore complete near $Q$ since $Q$ is compact and $\gamma$ is positive-definite. Since we are only interested in the dynamics of $\Phi_H$ near $Q$, we can regard $M\subset E$ as a neighborhood of the zero section, which we identify with $Q$. Thus, it follows from the definition of $\gamma$, that the following holds:
\begin{equation}
H(q,v)=H^\gamma(q,v)+o(|v|^2),\qquad H^\gamma(q,v):=\tfrac{1}{2}\gamma_q(v,v),
\end{equation}
for an $\varepsilon>0$ sufficiently small, the energy level $\Sigma_\varepsilon:=H^{-1}(\tfrac12\varepsilon^2)$ is a closed hypersurface, and the restriction of $\pi$ yields a sphere bundle $\pi\colon \Sigma_\varepsilon\to Q$. 

When $Q$ is symplectic, we can take $E$ to be the symplectic orthogonal of $TQ$ in $TM|_Q$. For every $q\in Q$, we denote by $\sigma_q$ the restriction of the symplectic form $\omega_q$ to $T_qQ$ and by $\rho_q$ the restriction of $\omega_q$ to $E_q$. Therefore, $\pi\colon (E,\rho)\to (Q,\sigma)$ is a symplectic vector bundle (with form $\rho$) over a symplectic manifold (with form $\sigma$). We denote
\begin{equation}
2m:=\dim Q,\qquad 2k:=\mathrm{rk}(E\to Q)
\end{equation}
The linearized Hamiltonian dynamics of $H$ at $Q$ is given by
\begin{equation}
\Phi_H^t(q,v)=(q,e^{tA_q}v),\qquad\forall\,t\in\mathbb R,\ \forall\,(q,v)\in E,
\end{equation}
where $A_q\colon E_q\to E_q$ is uniquely defined by
\begin{equation}\label{def_of_A}
\rho_q(w,A_qv)=\gamma_q(w,v),\qquad\forall\,w\in E_q
\end{equation}
and belongs to the Lie algebra $\mathfrak{sp}(\rho_q)$ of the $\rho_q$-symplectic linear group.
Since multiplying $H$ by a positive constant only changes the time parametrization of the Hamiltonian, we will assume the normalization
\begin{equation}\label{normalization_of_A_morse-bott}
\int_Q\frac{1}{\det A}\, \sigma^m=\int_Q\sigma^m=:\mathrm{Vol}_\sigma(Q).
\end{equation}
We now focus on low-energy levels for which the Hamiltonian flow is globally periodic.
\begin{definition}
A regular energy level $\Sigma_\varepsilon$ is called Besse if, up to a global smooth time reparametrization, the Hamiltonian flow of $H$ on the energy level is an almost everywhere free circle action. If the action is free, then we call the energy level Zoll.
\end{definition}
\begin{definition}
    We say that $H$ is Besse along a sequence of energies converging to the minimum if there exists a sequence of positive numbers $(\varepsilon_n)$ converging to $0$ such that $\Sigma_{\varepsilon_n}$ is Besse. Similarly, we define Hamiltonians $H$ that are Zoll along a sequence of energies converging to the minimum.
\end{definition}
In our main results, we will give necessary conditions for a Hamiltonian $H$ to be Zoll along a sequence of energies converging to a Morse--Bott symplectic minimum $Q$: Theorem \ref{thmA} in the general setting and Theorem \ref{thmB} in the magnetic setting. The symplecticity of $Q$ is a natural condition. Indeed, when the Hamiltonian flow of $H$ induces a circle action on a whole neighborhood of $Q$, then $Q$ is automatically symplectic \cite{McDuff_Salamon_intro}. Moreover, the key condition arising in our results concerns the relation between the fiberwise Hessian $\gamma$ and the fiberwise symplectic form $\rho$.

\begin{definition}
The metric $\gamma$ is called conformally $\rho$-compatible if $A=aJ$, where $a\colon Q\to\mathbb R$ is a positive function and $J$ is an almost-complex structure, that is, $J^2=-\mathrm{id}$. We say that $\gamma$ is $\rho$-compatible if $a=1$.
\end{definition}

We now turn to a crucial example showing that this compatibility condition naturally gives rise to Zoll dynamics on low-energy levels.

\subsection{A Crucial Example: The Coupling Form}\label{example}
Let $\nabla$ be any affine connection of $\pi\colon E\to Q$. The connection $\nabla$ yields an isomorphism $T_{(q,v)}E\cong T_qQ\oplus E_q$ made of a horizontal projection
\begin{equation}
T_{(q,v)}E\to T_qQ,\qquad \xi\mapsto \xi^\pi:=d_{(q,v)}\pi\xi
\end{equation}
and a vertical projection
\begin{equation}
T_{(q,v)}E\to E_q,\qquad \xi\mapsto \xi^\nabla:=\frac{\nabla}{dt}Z(0),
\end{equation}
where $Z: (-\delta,\delta) \rightarrow E$ such that $Z(0)=(q,v)$ and $\dot{Z}(0)=\xi$, and $\frac{\nabla}{dt}$ is the covariant derivative along the curve $(\pi\circ Z)$. The vertical and horizontal distributions are defined as
\begin{equation}
\mathcal V=\ker(\xi\mapsto \xi^\pi),\qquad \mathcal H=\ker(\xi\mapsto\xi^\nabla),
\end{equation}
so that there is a decomposition
\begin{equation}
T_{(q,v)}E=\mathcal H_{(q,v)}\oplus\mathcal V_{(q,v)}
\end{equation}
and we have the horizontal and vertical lift isomorphisms as inverses of the horizontal and vertical projections
\begin{equation}
T_qQ\to T_{(q,v)}E,\quad \zeta\mapsto \zeta^h,\qquad E_q\to T_{(q,v)}E,\quad w\mapsto w^v.
\end{equation}
Using the lifts, for every $(q,v)\in E$ we can define the horizontal differential $d^h_{(q,v)}K\in T_q^*Q$ and the vertical differential $d^v_{(q,v)}K\in E_q^*$ of a function $K\colon E\to \mathbb R$ by letting
\begin{equation}
\begin{aligned}
d^h_{(q,v)} K\zeta&:=d_{(q,v)} K(\zeta^h),\quad \forall\,\zeta\in T_qQ,\\
d^v_{(q,v)} Kw&:=d_{(q,v)}K(w^v),\quad \forall\,w\in E_q.
\end{aligned}
\end{equation}

After these preliminaries, we can define a symplectic form $\omega^{\sigma,\rho,\nabla}$ on the small neighborhood $M$ of $Q$ given by
\begin{equation}
\omega^{\sigma,\rho,\nabla}:=\pi^*\sigma+\tfrac12d\tau^{\rho,\nabla},
\end{equation}
where $\tau^{\rho,\nabla}$ is the angular or coupling one-form \cite{McDuff_Salamon_intro} given by 
\begin{equation}\label{def_of_coupling_1_form}
\tau^{\rho,\nabla}_{(q,v)}(\xi)=\rho_q(v,\xi^\nabla),\qquad\forall\, (q,v)\in E,\ \forall\, \xi\in T_{(q,v)}E.
\end{equation}
By the symplectic neighborhood theorem \cite{Weinstein_nbd_thm}, upon shrinking $M$ inside $E$, there is a symplectomorphism
\begin{equation}
\Psi^\nabla\colon (M,\omega^{\sigma,\rho,\nabla})\to (M',\omega)    
\end{equation}
to a neighborhood $M'$ of $Q$, fixing $Q$ pointwise.

In the horizontal-vertical splitting of $T_{(q,v)}E$ induced by $\nabla$, the symplectic form $\omega^{\sigma,\rho,\nabla}$ is written as
\begin{equation}\label{symplectic_form_matrix_morse-bott-case}
\omega^{\sigma,\rho,\nabla}_{(q,v)}=\begin{pmatrix}
\sigma_q+\tfrac12\rho_q(R^\nabla_q(\cdot,\cdot)v,v)&\tfrac12(\nabla_{\cdot\,}\rho)(v,\cdot)\\ -\tfrac12(\nabla_{\cdot\,}\rho)(v,\cdot)& \rho_q
\end{pmatrix},
\end{equation}
where $R^\nabla_q$ is the curvature of $\nabla$.

Let us now suppose that
\begin{equation}
\text{$\gamma$ is $\rho$-compatible},\qquad\nabla\rho=0,\qquad\nabla \gamma=0.    
\end{equation}
In particular, $\rho(\cdot,J\cdot)=\gamma$ and 
\begin{equation}\label{symplectic_form_matrix_morse-bott-case_b}
\omega^{\sigma,\rho,\nabla}_{(q,v)}=\begin{pmatrix}
\sigma_q+\tfrac12\rho_q(R^\nabla_q(\cdot,\cdot)v,v)&0\\ 0& \rho_q
\end{pmatrix},
\end{equation}
If we assume, in the new coordinates given by $\Psi^\nabla$, that $H=H^\gamma$, then 
\begin{equation}
d^hH=0,\qquad d^vH=\gamma(\cdot,v)=\rho(\cdot, Jv).
\end{equation}
Therefore, 
\begin{equation}
X_H(q,v)=(J_qv)^v,\qquad \Phi_H^t(q,v)=(q,e^{tJ_q}v)
\end{equation}and $\Sigma_\varepsilon$ is Zoll for every $\varepsilon$ sufficiently small. Up to symplectomorphisms, these are the only examples of Hamiltonians which are Zoll at all small energy levels by the Marle \cite{Marle} equivariant symplectic neighborhood theorem \cite[Corollary II.1.12 and Remark II.1.13]{Audin}. Indeed, if $H$ is such a Hamiltonian on $(M,\omega)$, then the symplectomorphism $\Psi^\nabla$ with $\nabla $ can be chosen to intertwine the Hamiltonians $H\circ\Psi^\nabla=H^\gamma$, where $\gamma$ is $\rho$-compatible, $\nabla\rho=0$ and $\nabla\gamma=0$.  

\subsection{The Main Theorem in the Symplectic Setting}
Our first main theorem shows that example contained in Section \ref{example} is universal among systems that are Zoll along a sequence of energies converging to the Morse--Bott minimum.
\begin{bigthm}\label{thmA}
Let $H$ be a smooth Hamiltonian on a symplectic manifold $(M,\omega)$ having a Morse--Bott minimum at a connected, closed submanifold $Q\subset M$, and satisfying the normalization \eqref{normalization_of_A_morse-bott}. If $H$ is Zoll along a sequence of energies converging to $\min H$, then the fiberwise Hessian $\gamma$ of $H$ along the symplectic normal bundle at $Q$ is $\rho$-compatible, where $\rho$ is the restriction of $\omega$ to the normal bundle.
\end{bigthm}
Let us give a brief sketch of the proof in three steps. Assume that $H$ is Zoll along a sequence of energies converging to the minimum. \medskip

\textit{Step 1.} There exists $T_*>0$ such that, up to a global time reparametrization independent of $n$, all periods of all periodic orbits with energy $\varepsilon_n$ converge to $T_*$ as $n$ goes to infinity, see Theorem \ref{t:T_n}. This result combines two facts. First, there exists a periodic orbit on every low energy level whose period is uniformly bounded \cite{usher}. Second, as $n$ goes to infinity, the dynamics converges to the fiberwise linear flow generated by $A$ and, up to a global time reparametrization, all periods of this fiberwise linear flow belong to a nowhere dense set.

\textit{Step 2.} First, we use Step 1 to show that the fiberwise linearized flow generated by $A$ is Besse. If it were not Zoll, then there exists a periodic submanifold $W$ strictly contained in $E$ made of orbits with minimal period. Building on the work of Bottkol \cite{Bottkol} and Kerman \cite{Kerman}, we show that periodic orbits bifurcate from $W$ for all small energy. This would contradict the existence of $T_*$ in Step 1.

\textit{Step 3.} By Step 2, $\gamma$ is conformally $\rho$-compatible, that is, $A=aJ$ for some function $a\colon Q\to(0,\infty)$ and we need to show that $a$ is constant. We use coordinates in which $\omega=\omega^{\sigma,\rho,\nabla}$. Since $\rho$ and the fiberwise Hessian $\gamma$ do not depend on the chosen symplectic connection $\nabla$, we can assume that the connection satisfies $\nabla \rho=0$ and $\nabla J=0$. In these special coordinates, we show that at every regular point $q\in Q$ of the function $a$ the Hamiltonian flow of $H$ has a slow horizontal drift in the direction of the Hamiltonian flow of the function $a$ on $(Q,\sigma)$. This is again a contradiction to the existence of $T_*$ in Step 1. Thus the proof sketch of Theorem \ref{thmA} is complete.
\medskip

Theorem \ref{thmA} shows that being Zoll along a sequence of energies converging to the minimum implies that $\gamma$ is $\rho$-compatible. In this case, we can choose $\nabla$ such that both $\nabla\rho=0$ and $\nabla \gamma=0$ hold. Therefore, the Hamiltonian $H^\gamma$, which is the quadratic part of $H$, gives a circle action on $M$ that is free outside the zero section and has the zero section as fixed-point set, see Example \ref{example}.

Thus it should be possible to compute a fiberwise Birkhoff normal form for $H$ in powers of $v$. The Zoll condition will impose restrictions on the terms of the Birkhoff normal form. In the general setting considered here, this is likely to be a delicate task. There is, however, a case of great physical interest for which the lowest order of the Birkhoff normal form can be explicitly computed and yields interesting geometric information. This is the case of symplectic magnetic systems and will be presented in the next subsection.
\subsection{The Magnetic Setting}
An important example of the symplectic setting of the previous subsection is given by magnetic systems $(Q,g,\beta)$ \cite{Arn}. Here, $Q$ is a closed and connected manifold, $g$ is a Riemannian metric on $Q$, and $\beta$ is a closed 2-form on $Q$ referred to as the magnetic form. The form $\beta$ yields a symplectic form $\omega_{\mathrm{can},\beta}$ on the cotangent bundle $\pi\colon T^*Q\to Q$ given by 
\begin{equation}
\omega_{\mathrm{can},\beta}:=d\lambda-\pi^*\beta,
\end{equation}
where $\lambda$ is the canonical 1-form defined by 
\begin{equation}
\lambda_{(q,p)}(u)=p(d_{(q,p)}\pi \cdot u),\qquad \forall\,(q,p)\in T^*Q,\ \forall\,u\in T_{(q,p)}T^*Q.    
\end{equation}
The metric $g$ yields a kinetic Hamiltonian
\begin{equation}
H^g\colon T^*Q\to\mathbb R,\qquad H^g(q,p)=\tfrac12g_q(p,p),
\end{equation}
where we also denote by $g$ the dual metric on the cotangent bundle.

To ease our geometric intuition, we will pull back the symplectic form $\omega_{\mathrm{can},\beta}$ and the Hamiltonian $H^g$ to the tangent bundle $TQ$ using the isomorphism given by the metric $g$. We denote the objects with the same symbols. Notice that on $TQ$ we have
\begin{equation}
\lambda_{(q,v)}(u)=g_q(v,d_{(q,v)}\pi \cdot u),\qquad \forall\,(q,v)\in TQ,\ \forall\,u\in T_{(q,v)}TQ.
\end{equation}
On $TQ$, the flow lines of $H$ at energy $\tfrac12\varepsilon^2$ are $\Phi_H^t(q(0),\dot q(0))= (q(t),\dot q(t))$ where $q\colon\mathbb R\to Q$ are curves with speed $\varepsilon$ and satisfying the magnetic geodesic equation
\begin{equation}
\frac{\nabla^g}{dt}\dot q=B_q\dot q,
\end{equation}
where $\nabla^g$ is the Levi-Civita connection of $g$ and $B_q\colon T_qQ\to T_qQ$ is the Lorentz endomorphism defined by
\begin{equation}
g_q(B_q u,v)=\beta_q(u,v),\qquad\forall\,q\in Q,\ \forall\,u,v\in T_qQ.
\end{equation}
The existence problem of periodic magnetic geodesics on low energy levels is the subject of a vast and beautiful literature, see \cite{NT,Gin96,CMP,Merry0,Merry1,Merry2,Schneider,AB,CZ,MG} for some milestones.
The function $H^g$ has a Morse--Bott minimum at the zero section $Q\subset TQ$, which is a symplectic submanifold for $\omega_{\mathrm{can},\beta}$ if and only if $\beta$ is symplectic on $Q$. Thus, in this situation one can apply the work on periodic orbits near Morse--Bott minima that we discussed in the previous subsection, see the literature cited in Section \ref{s:setting} and Theorem \ref{thmA}. In fact, magnetic systems have been a source of impulse for studying dynamics near symplectic submanifolds.\medskip

From now on, we will assume that $\beta$ is symplectic. In this case the restriction of $\omega_{\mathrm{can},\beta}$ to $Q$ is $\sigma=-\beta$.  The $\omega_{\mathrm{can},\beta}$-orthogonal to $T_qQ$ inside $T_{(q,0)}TQ$ is expressed in the horizontal-vertical splitting as the image of the map \begin{equation}\label{j-map}
\jmath_q\colon T_qQ\to T_{(q,0)}TQ,\quad \jmath_q(u)=(u,B_qu).
\end{equation}
Using the map $\jmath_q$ to parametrize the orthogonal space to $T_qQ$, we see that
\begin{itemize}
    \item the fiberwise Hessian of $H^g$ is $\gamma=B^*g$;
    \item the fiberwise symplectic form is $\rho=\beta$;
    \item the corresponding generator of the linear flow is $A=B$;
    \item the normalization \eqref{normalization_of_A_morse-bott} becomes \begin{equation}\label{e:normalmag}
\mathrm{Vol}_g(Q)=\mathrm{Vol}_\beta(Q).
\end{equation}
\end{itemize}
By Section \ref{example}, for every affine connection $\nabla$ on $TQ$, there is a symplectomorphism
\begin{equation}\label{weinstein symplectomorphism}
\Psi^\nabla\colon (M,\omega^{-\beta,\beta,\nabla})\to (M',\omega_{\mathrm{can},\beta})    
\end{equation}
between two neighborhoods $M$ and $M'$ of the zero section with $d_{(q,0)}\Psi (0,u)=\jmath_q(u)$.

\subsection{A Crucial Example: Complex Space Forms}\label{example_magnetic}
The definition of magnetic systems does not require any compatibility between $g$ and $\beta$. However, there are some special classes of magnetic systems that deserve special attention. If $g$ is $\beta$-compatible, that is, $B=J$ for some almost complex structure $J$, then $(g,\beta)$ is said to be almost Kähler. When $J$ is integrable, then $(g,\beta)$ is a Kähler structure. Kähler manifolds $Q$ for which the geodesic reflection at every point extends to a global holomorphic involution of $Q$ are called Hermitian symmetric spaces. For these spaces, \cite{B24} yields a precise description of the map $\Psi^{\nabla^g}$, of the Hamiltonian $H^g\circ\Psi^{\nabla^g}$, and of the neighborhoods $M$ and $M'$ in \eqref{weinstein symplectomorphism} in terms of the holomorphic sectional curvature of $\nabla^g$. If $SQ\to Q$ is the unit sphere bundle of $TQ$ with respect to $g$ and $\pi\colon P_{\mathbb C}(TQ)\to Q$ is the complex projectivization of $TQ$ whose fibers $P_{\mathbb C}(T_qQ)$ are the complex lines contained in $T_qQ$, then the holomorphic sectional curvature $K^{g,J}\colon P_{\mathbb C}(TQ)\to\mathbb R$ is defined as
\begin{equation}
K^{g,J}(\mu_q(v)):=g_q(R^{\nabla^g}_q(v,J_qv)J_qv,v),\quad \forall\,v\in S_qQ,
\end{equation}
where $\mu_q\colon S_qQ\to P_{\mathbb C}(T_qQ)$ assigns to a vector the complex line that contains it. Kähler manifolds for which the holomorphic sectional curvature is constant are locally Hermitian symmetric spaces, their universal cover being isomorphic to one of the three complex space forms: the euclidean space $\mathbb C^m$, the complex projective space $\mathbb C P^m$, the complex hyperbolic space $\mathbb C H^m$ \cite[Ch.\ IX.7]{KN69}. In this case, Bimmermann showed that the flow of $H^g$ on $(M',\omega_{\mathrm{can},\beta})$ is Zoll at every energy level with period given by
\begin{equation}
T(H^g)=\frac{2\pi}{\sqrt{1+2\kappa H^g}},\qquad \kappa:=K^{g,J}\in\mathbb R.    
\end{equation}
Since the flow of $H^g$ is Zoll on $(M,\omega^{-\beta,\beta,\nabla^g})$ at every energy level with period $2\pi$ by Example \ref{example}, and $H^g\circ \Psi^{\nabla^g}=f(H^g)$ for some function $f$, we also get $\frac{d f}{dH}T(H)=2\pi$, which yields
\begin{equation}
H^g\circ\Psi^{\nabla^g}=H^g+\frac{\kappa}{2}(H^g)^2.
\end{equation}
\subsection{The Main Theorem in the Magnetic Setting}
In our second main theorem, we build on Theorem \ref{thmA} and explore the universality of the example in Section \ref{example_magnetic} among symplectic magnetic systems which are Zoll along a sequence of energies converging to zero. For this purpose, we recall the definition of the Chern connection $\nabla^{g,J}$ for an almost Hermitian structure $(g,J)$. Being Hermitian means that $g$ is a Riemannian metric on $Q$ and $J$ is an almost complex structure on $Q$ such that $\beta:=g(J\,\cdot\,,\,\cdot\,)$ is a 2-form, albeit not necessarily closed. The Chern connection is then defined as the only Hermitian connection whose torsion $T^{g,J}$ has vanishing $(1,1)$-component, that is,
\begin{equation}
\nabla^{g,J}g=0,\qquad\nabla^{g,J}J=0,\qquad T^{g,J}(J\,\cdot\,,\,\cdot\,)=T^{g,J}(\,\cdot\,,J\,\cdot\,).
\end{equation}
When $(g,J)$ is almost Kähler, that is, $\beta$ is closed and $(Q,g,\beta)$ is a magnetic system, then by \cite[Section~2.1]{McDuff_Salamon_jhol}
\begin{equation}
T^{g,J}=-\frac{1}{4}N^J,
\end{equation}
where $N^J$ is the Nijenhuis tensor
\begin{equation}
N^J(X,Y):=[X,Y]+J[JX,Y]+J[X,JY]-[JX,JY],
\end{equation}
which vanishes if and only if $J$ is integrable \cite{NN57}. We interpret the Nijenhuis tensor $N^J\in \Omega^1(Q,\mathrm{End}(TQ))$ as a one-form on $Q$ with values in the endomorphisms of $TQ$ by setting
\begin{equation}
N^J_v:=N^J_q(v,\cdot)\colon T_qQ\to T_qQ,\qquad \forall\,q\in Q,\ \forall\,v\in T_qQ.
\end{equation}
We denote by $(N^J_v)^*\colon T_qQ\to T_qQ$ the adjoint of $N^J_v$ with respect to $g$. If $\beta$ is closed, then $N^J_v$ is antisymmetric for all $v$ if and only if $N^J=0$, that is, $J$ is integrable. On the other hand, the anticomplex relationships
\begin{equation}
N_{Jv}=-JN_v=-N_vJ,\qquad N_{Jv}^*=-JN_v^*=-N_v^*J.
\end{equation}
hold and tell us that the function
\begin{equation}
|(N^J)^*|^2\colon P_\mathbb C(TQ)\to\mathbb R,\qquad |(N^J)^*|^2(\mu_q(v)):=|(N^J_v)^*v|^2
\end{equation}
is well defined.
\begin{bigthm}\label{thmB}
Let $(Q,g,\beta)$ be a magnetic system with $\beta$ symplectic satisfying the normalization \eqref{e:normalmag}. If $H^g$ is Zoll along a sequence of energies converging to zero, then $g$ is $\beta$-compatible, that is, $B=J$ is an almost complex structure. Moreover, there exists a constant $\kappa$ such that
\begin{equation}\label{curvature torsion constant}
K^{g,J}-\tfrac{1}{24}|(N^J)^*|^2=\kappa,
\end{equation}
where $K^{g,J}$ and $N^J$ are the holomorphic sectional curvature of the Chern connection and the Nijenhuis tensor. If $J$ is integrable, then $(g,J)$ is a complex space form and the magnetic flow is Zoll for all low energy values. 
\end{bigthm}
\begin{bigcor}
A conformally Kähler magnetic system $(Q,g,\beta)$ satisfying the normalization \eqref{e:normalmag} is Zoll along a sequence of energies converging to zero if and only if $B=J$ and $(g,J)$ is a complex space form.\hfill\qed
\end{bigcor}
\begin{remark} It is unknown to us if there are examples of almost Kähler structures $(g,J)$ such that \eqref{curvature torsion constant} holds but $J$ is not integrable. Moreover, we don't know if given \eqref{curvature torsion constant}, the system is Zoll at every low energy level like in the integrable case. To better understand this problem, one could compute the next order in the expansion of the magnetic flow for low energies, when \eqref{curvature torsion constant} holds. Imposing the next order to be constant would give additional information on the almost Kähler structure $(g,J)$. Finally, we notice that a scalar version of the quantity above appears in \cite[Theorem 0.1]{Don00} with the Levi--Civita connection involved instead of the Chern connection. 
\end{remark}
Let us give a sketch of the proof of Theorem \ref{thmB} in two steps. Assume that $(g,\beta)$ is Zoll along a sequence of energies converging to zero. By Theorem \ref{thmA} and the identification given by the map $\jmath$ below \ref{j-map}, we deduce that $g$ is $\beta$-compatible. Therefore, $B=J$ and we can choose as symplectic connection the Chern connection $\nabla^{g,J}$. At this point, an option would be to compute higher orders in $v$ of the map $\Psi^{\nabla^{g,J}}$ to compute the first non-constant order of $H^g\circ\Psi^{\nabla^g}$. Instead, we choose a different route generalizing the method of \cite[Theorems~1.1 and~1.9]{AB_normal_forms_surfaces}, which proves Theorem~\ref{thmB} when $Q$ is a surface. We proceed in two steps.

\textit{Step 1.} For every $\varepsilon>0$, we will consider the restriction of $\omega_{\mathrm{can},\beta}$ to the energy level $\Sigma_\varepsilon$. The restricted two-form has a one-dimensional kernel which is generated by the magnetic vector field on $\Sigma_\varepsilon$. By rescaling $SQ\to \Sigma_\varepsilon$, $(q,v)\to(q,\varepsilon v)$, we pull back the restricted form to the form
\begin{equation}
\omega_\varepsilon:=\varepsilon d\lambda-\pi^*\beta
\end{equation}
on $SQ$. Using a Moser argument, we find a family of diffeomorphisms $\psi_\varepsilon\colon SQ\to SQ$ with $\psi_0=\mathrm{id}_{SQ}$ such that
\begin{equation}
\psi_\varepsilon^*\omega_\varepsilon=-\pi^*\beta+d\Big((H_\varepsilon\circ\mu)\tau\Big)+o(\varepsilon^4),
\end{equation}
where $\tau=\tau^{\beta,\nabla^{g,J}}$ is the angular form defined in \eqref{def_of_coupling_1_form}, $\mu\colon SQ\to P_\mathbb C(TQ)$ is the circle-bundle projection, and $H_\varepsilon\colon P_\mathbb C(TQ)\to \mathbb R$ is the function
\begin{equation}
H_\varepsilon:=\frac{\varepsilon^2}2+\frac{\varepsilon^4}{4}\frac{\hat K}{2},\qquad \hat K:=K^{g,J}-\tfrac{1}{24}|(N^J)^*|^2. 
\end{equation}

\textit{Step 2.} The form $\tau$ is the connection one-form of $\mu$. In particular $d\tau=\mu^*\delta$ for a curvature two-form $\delta$ on $P_\mathbb C(TQ)$. In the horizontal-vertical splitting, we have
\begin{equation}\label{eq:delta}
\delta_{\mu(q,v)}=\begin{pmatrix}
\rho_q(R_q^{\nabla^{g,J}}(\cdot,\cdot)v,v)&0\\
0&2\bar\beta_{\mu(q,v)}
\end{pmatrix},
\end{equation}
where $\bar\beta_{\mu(q,v)}$ is the Fubini-Study symplectic form induced by the fiberwise symplectic form $\beta_q$ at $\mu(q,v)$. These facts imply that, after performing the change of coordinates $\psi_\varepsilon$, the projection of the magnetic flow onto $P_\mathbb C(TQ)$ is generated, up to $o(\varepsilon^4)$, by the Hamiltonian vector field of $H_\varepsilon$ with respect to the symplectic form
\begin{equation}
-\pi^*\beta+H_\varepsilon \delta.
\end{equation}
We can decompose $X_{H_\varepsilon}$ in the horizontal-vertical splitting given by $\nabla^{g,J}$ on $T(P_\mathbb C(TQ))$ as
\begin{equation}
X_{H_\varepsilon}^\pi=\varepsilon^4Y+o(\varepsilon^4),\qquad X_{H_\varepsilon}^\nabla=\varepsilon^2Z+o(\varepsilon^2),
\end{equation}
where the vector fields $Y$ and $Z$ are defined by
\begin{equation}
\beta(Y,\cdot)=-\frac{1}{8}d^h \hat K,\qquad \bar\beta(Z,\cdot)=+\frac{1}{8}d^v\hat K.
\end{equation}
This means that the projected dynamics slowly drifts with speed $\varepsilon^2$ in the vertical direction and with speed $\varepsilon^4$ in the horizontal direction. Using the existence of $T_*$ from Step 1 in the proof of Theorem \ref{thmA} twice, we first deduce that $d^v\hat K=0$ and then $d^h\hat K=0$. Hence $\hat K$ is constant and the proof sketch is complete.

\subsection{Plan of the Paper}

In Section~\ref{sec:prelim-thmA} we introduce the local symplectic model near the Morse--Bott minimum and collect the preliminaries needed for the proof of Theorem~\ref{thmA}. Section~\ref{sec:limit-period} proves the existence of a limit period for Zoll energy levels converging to the minimum. In Section~\ref{sec:bifurcation-conformality}, we show that the fiberwise Hessian is conformally $\rho$-compatible. This relies on a bifurcation theorem in the case where the linearized flow at the minimum is Besse. In turn, the bifurcation theorem relies on a normal form à la Bottkol, which is proved in Section~\ref{sec:proof-bifurcation}. Section~\ref{sec:conformally-compatible} completes the proof of Theorem~\ref{thmA} by showing that the function relating the fiberwise Hessian and $\rho$ is constant. The proof of Theorem~\ref{thmB} is carried out in Sections~\ref{sec:proof-thmB} and~\ref{sec:drift-analysis}. In Section~\ref{sec:proof-thmB}, we derive a normal form for the low-energy magnetic dynamics. In Section~\ref{sec:drift-analysis}, we use this normal form to analyze the magnetic flow and deduce the curvature rigidity statement in the magnetic setting.

\subsection{Acknowledgements}
G.~B.~gratefully acknowledges support from the Simons Center for Geometry and
Physics, Stony Brook University at which some of the research for this paper was performed during the program “Contact Geometry, General Relativity and Thermodynamics”. G.~B.~warmly thanks Francesco Lin for inspiring discussions around \cite{Don00}. J.~B.~was supported by the Engineering and Physical Sciences Research Council [grant number EP/Z535977/1]. S.S is supported by the DFG through the project SFB/TRR 191 "Symplectic Structures in Geometry,
Algebra and Dynamics", Projektnummer 281071066-TRR 191.
\subsection{Notation}
For a family of objects parametrized by some $\varepsilon>0$, we use the notation $O(\varepsilon^d)$ for $d\geq0$ to mean that the $C^k$-norms are $O(\varepsilon^d)$ as $\varepsilon\to 0$ for some $k\geq 0$. The parameter $k$ will not be explicitly written, and it might depend on the object. However, this does not represent an issue since we assume that the Hamiltonian and symplectic manifold are smooth, and the theorems hold when $k$ is sufficiently large.
\section{Theorem \ref{thmA}, Preliminaries}\label{sec:prelim-thmA}\subsection{The Morse--Bott Lemma}
We begin by recalling the setup of Theorem \ref{thmA}. Let  $(M,\omega)$ be a symplectic manifold with a smooth Hamiltonian $H\colon M\to\mathbb [0,\infty)$ which has a compact, connected symplectic Morse--Bott minimum at $Q:=H^{-1}(0)\subset M$. We denote with $\sigma$ the restriction of $\omega$ to $Q$. Let $\pi\colon E\to Q$ be the symplectic normal bundle of $Q$ and denote with $\rho$ the restriction of $\omega$ to the normal bundle $E\subset TM|_Q$. We denote by $\gamma$ the Hessian of $H$ along $E$ and let
\begin{equation}
\Sigma:=(H^{\gamma})^{-1}(\tfrac12)=\{(q,v)\in E\ |\ \gamma_q(v,v)=1\}.
\end{equation}
We suppose that $M\subset E$ is a small neighborhood of $Q$. In the following we will fiberwise expand several objects in terms of $v$ around the zero section, that is, at $v=0$.

By the symplectic neighborhood theorem, upon changing coordinates, we can assume that $\omega=\omega^{\sigma,\rho,\nabla}$, where $\nabla$ is any affine connection on $E$, see \eqref{symplectic_form_matrix_morse-bott-case}. In particular, \begin{equation}\label{e:sympnabla'}
\omega_{(q,v)}=
\begin{pmatrix}
\sigma_q+O(|v|^2)&O(|v|)\\
O(|v|)&\rho_q
\end{pmatrix}.
\end{equation}
Since $H(q,v)=\tfrac{1}{2}\gamma_q(v,v)+O(|v|^3)$, we can use the formula \eqref{e:sympnabla'} for $\omega$ to expand the horizontal and vertical projections of $X_H(q,v)$ in powers of $v$ as
\begin{equation}
X_H^\pi(q,v)=\zeta_q(v,v)+O\big(|v|^3\big),\qquad X_H^{\nabla}(q,v)=A_qv+O\big(|v|^2\big),
\end{equation}
where $A_qv$ is the linear map defined in \eqref{def_of_A} and $\zeta_q$ is a quadratic form. In the special case in which $\nabla$ preserves $\rho$, formula \eqref{symplectic_form_matrix_morse-bott-case} reduces to \eqref{symplectic_form_matrix_morse-bott-case_b} and we deduce that $\zeta_q(v,v)$ is given by
\begin{equation}\label{eq:zeta-formula}
\sigma_q(\zeta_q(v,v),\eta)=\tfrac12(\nabla_\eta \gamma)_q(v,v),\qquad\forall\,\eta\in T_qQ.
\end{equation}
We fix the notation for the vector field $X_0$ on $E$ given by
\begin{equation}
X_0(q,v):=(A_qv)^v.
\end{equation}

By the Morse--Bott Lemma \cite{Bott,BH} there exists an embedding $F\colon M\to E$ with
\begin{equation}
F(q,v)=\big(q,v+O(|v|^2)\big),\qquad H\circ F=H^\gamma.
\end{equation}
Differentiating this formula, we get
\begin{equation}\label{e:dF}
dF=\begin{pmatrix}
1&0\\
O(|v|^2)&1+O(|v|)
\end{pmatrix}
\end{equation}
in the horizontal-vertical decomposition given by $\nabla$.

\begin{lemma}\label{l:F^*X_H}
Let $F^*X_H$ be the pullback of the Hamiltonian vector field. There exists a vector field $Y$ such that
\begin{equation}
F^*X_H=X_0+Y,
\end{equation}
satisfying the estimates
\begin{equation}
Y^\pi=\zeta_q(v,v)+O(|v|^3)=O(|v|^2),\qquad Y^{\nabla}=O(|v|^2).
\end{equation}
\end{lemma}
\begin{proof}
By definition, $F^* X_H=dF^{-1}X_H(F)$. We have
\begin{equation}
dF^{-1}=\begin{pmatrix}
1&0\\
O(|v|^2)&1+O(|v|)
\end{pmatrix}.
\end{equation}
Thus, for all $(q,v)\in\Sigma$, we get
\begin{equation}
\begin{aligned}
(F^* X_H)^\pi(q,v)&=\zeta_q(v,v)+O(|v|^3),\\
(F^*X_H)^{\nabla}(q,v)&=A_qv+O(|v|^2).
\end{aligned}
\end{equation}
\end{proof}
The vector field $F^*X_H$ is the Hamiltonian vector field of $H^\gamma$ for the symplectic form $F^*\omega$. For later purposes, we compute this pull-back form.
\begin{lemma}\label{order-of_F_lemma}
The pull-back form $F^*\omega$ is in the same de Rham cohomology class as $\pi^*\sigma$ and it has the expression
\begin{equation}\label{e:Fomega}
F^*\omega=\begin{pmatrix}
        \sigma+O(|v|) &O(|v|) \\
        O(|v|) & \rho+O(|v|)
    \end{pmatrix}.
\end{equation}
with respect to the horizontal-vertical splitting of $\nabla$.
\end{lemma}
\begin{proof}
Since $\omega$ is cohomologous to $\pi^*\sigma$ and $F^*\pi^*\sigma=\pi^*\sigma$, we see that $F^*\omega$ is also cohomologous to $\pi^*\sigma$.
Using equations \ref{e:sympnabla'} and \eqref{e:dF}, we get
\begin{equation*}
F^*\omega=\begin{pmatrix}
1&O(|v|^2)\\
0&1+O(|v|)
\end{pmatrix}\begin{pmatrix}
 \sigma+O(|v|^2)&O(|v|)\\ O(|v|)&\rho   
\end{pmatrix}\begin{pmatrix}
1&0\\
O(|v|^2)&1+O(|v|)
\end{pmatrix}.
\end{equation*}
Performing the matrix multiplication yields the desired equation.
\end{proof}
We showed that $F^*X_H=X_0+O(|v|)$, where $X_0(q,v)=(A_qv)^v$. Since $A_q$ satisfies \eqref{def_of_A} and $\gamma_q$ is positive definite, we conclude that all the eigenvalues of $A_q$ are purely imaginary and non-zero (in particular $\det A_q$ is positive). This means that the dynamics of $X_0$ decomposes as a direct sum of rotations with different speeds. To normalize these speeds independently of $q$, we perform the smooth rescaling
\begin{equation}
{\tilde X}_{H}(q,v):=(\det A_q)^{-\frac{1}{2k}}X_{H}(q,v),
\end{equation}
where $2k$ is the rank of $\pi\colon E\to Q$. If we let
\begin{equation}
\tilde A_q:=(\det A_q)^{-\frac{1}{2k}}A_q,
\end{equation}
and define
\begin{equation}
\tilde X_0(q,v):=(\tilde A_qv)^v,\qquad \tilde\tau_{(q,v)}:=(\det A_q)^{\frac{1}{2k}}\tau_{(q,v)},
\end{equation}
then we get
\begin{equation}\label{e:F^*XHtilde}
F^*{\tilde X}_{H}=\tilde X_0+\begin{pmatrix}
(\det A_q)^{-\frac{1}{2k}}\zeta_q(v,v)+O(|v|^3)\\
O(|v|^2)
\end{pmatrix}=O(|v|^2)
\end{equation}
and the normalizations
\begin{equation}
\tilde\tau_{(q,v)}(\tilde X_0(q,v))=1,\qquad \det \tilde A_q=1,\qquad\forall\,(q,v)\in\Sigma.
\end{equation}
Using the normalization, we see that for every $q\in Q$ the eigenvalues of $\tilde A_q$ are given by \begin{equation}
        \big\{\pm i\tilde a_1(q),\, ,\, \cdots,\, \pm i\tilde a_k(q)\big\}.
    \end{equation}
where $\tilde a_1(q),\ldots,\tilde a_k(q)$ are positive real numbers, depend continuously on $q$ and satisfy
    \begin{equation}\label{normal_ak}
        \tilde a_1(q)\geq \cdots \geq \tilde a_k(q),\qquad \tilde a_1(q)\cdot\ldots\cdot\tilde a_k(q)=1.
    \end{equation}
We call $\tilde a_1(q),\ldots\tilde a_k(q)$ the spectral numbers of $\tilde A_q$. Since some of the spectral numbers might coincide, we define the multiplicities
\begin{equation}
k_{\tilde a}(q):=\big\{\text{number of times $\tilde a$ is among $\tilde a_1(q),\ldots,\tilde a_k(q)$}\big\},\quad\forall\,\tilde a>0.
\end{equation}

We can choose complex coordinates $(z_1,\ldots,z_k)\in \mathbb C^k\cong E_q$ such that the flow of $\tilde X_0$ on $E_q$ is given by
\begin{equation}\label{flow_fully_periodic}
\Phi_{\tilde X_0}^t(z_1,\ldots,z_k)=\big(e^{it\tilde a_1(q)} z_1,\ldots,e^{it\tilde a_k(q)}z_k\big).
\end{equation}
The normalization allows us to say that when the flow of $\tilde X_0$ on $E_q$ is fully periodic, then the periods of non-constant orbits belong to the countable set independent of $q$ given by
\begin{equation}
2\pi\cdot(\mathbb{Q}_+^k)^{\frac{1}{k}}:= \Big\{2\pi\Big(\prod^k_{i=1} r_i\Big)^{\frac{1}{k}} \; \Big\vert \; r_1,\ldots, r_k \in \mathbb{Q}_+ \Big\}
\end{equation}
\begin{lemma}\label{nowhere_dense_lemma_MB}
Let $q\in Q$ be such that the flow of $\tilde X_0$ on $E_q$ is periodic. If $T$ is the period of any periodic orbit of $\tilde X_0$ on $E_q\setminus\{0_q\}$, then
    \begin{equation}
        T \in 2\pi\cdot (\mathbb{Q}_+^k)^{\frac{1}{k}}.
    \end{equation}
\end{lemma}
\begin{proof}
If all of the orbits of $\tilde X_0$ on $E_q$ are periodic, then by \eqref{flow_fully_periodic} there exists a positive real number $t(q)$ and $n_1(q),n_2(q),\cdots n_k(q) \in \mathbb{N}$ such that
    \begin{equation}
        t(q) \tilde a_j(q) = 2\pi n_j(q),\qquad \forall\,j=1,\ldots,k.
    \end{equation}
    Taking the geometric mean of all these equations and using the normalization \eqref{normal_ak} we get
    \begin{equation}
    t(q)= 2\pi \Big(\prod^k_{i=1}n_i(q)\Big)^{\frac{1}{k}}
    \end{equation}
    Therefore,
    \begin{equation}\label{expression_for_aj}
        \tilde a_j(q)= \frac{n_j(q)}{\big(\prod^k_{i=1}n_i(q)\big)^{\frac{1}{k}}}.
    \end{equation}
    Any periodic orbit of $\tilde X_0$ on $E_q\setminus\{0_q\}$ will have a non-zero component in the $j$-th complex factor for some $j=1,\ldots,k$, see \eqref{flow_fully_periodic}. Thus, if $T$ is the period of the orbit, there is $d\in \mathbb{N}$ such that the following holds:
    \begin{equation}
        T\tilde a_j(q) = 2\pi d
    \end{equation}
    Using \eqref{expression_for_aj} we conclude
    \begin{equation}
        T=2\pi\frac{d}{\tilde a_j(q)}=2\pi \Big(\prod^k_{i=1}\frac{d n_i(q)}{n_j(q)}\Big)^{\frac{1}{k}}\in2\pi\cdot (\mathbb{Q}_+^k)^{\frac{1}{k}}.\qedhere
    \end{equation}
\end{proof}
\subsection{Rescaling to the Unit Bundle}
We take now an affine connection that satisfies 
\begin{equation}
\nabla\gamma=0.
\end{equation}
Let $\varepsilon>0$ and define
\begin{equation}
\Sigma_\varepsilon:=(H^{\gamma})^{-1}(\tfrac12\varepsilon^2).    
\end{equation}
Denote by
\begin{equation}
\Sigma:=\Sigma_1=\{(q,v)\in E\ |\ \gamma_q(v,v)=1\}
\end{equation}
the unit sphere bundle of $E$. Since $\nabla\gamma=0$, for every $\varepsilon>0$ and every $(q,v)\in \Sigma_\varepsilon$, we obtain the splitting
\begin{equation}\label{e:split}
T_{(q,v)}\Sigma_\varepsilon=\mathcal H_{(q,v)}\oplus T_v\Sigma_{\varepsilon,q},
\end{equation}
where $\mathcal H$ is the horizontal distribution of $\nabla$. Since $H^\gamma$ is two-homogeneous, we have a well-defined rescaling 
\begin{equation}
I_\varepsilon\colon \Sigma\to \Sigma_\varepsilon,\qquad I_\varepsilon(q,v)=(q,\varepsilon v)
\end{equation}
which preserves the splitting \eqref{e:split}, and we get
\begin{equation}\label{e:scaling}
d_{(q,v)}I_\varepsilon=\begin{pmatrix}
1&0\\
0&\varepsilon
\end{pmatrix}.
\end{equation}
By the definition of the map $F$, we have 
\begin{equation}
\Sigma_\varepsilon=F^{-1}\big(H^{-1}(\tfrac12\varepsilon^2)\big).
\end{equation}
Therefore, we can pull back the vector field $\tilde X_H$ and the two-form $\omega$ from $H^{-1}(\tfrac12\varepsilon^2)$ to $\Sigma$ using the map $F\circ I_\varepsilon$ and get the following expansion.
\begin{lemma}\label{l:tildeX}
There is a vector field $\tilde Y_\varepsilon$ on $\Sigma$ such that the vector field
\begin{equation}
\tilde X_\varepsilon:=\tilde X_0+\tilde Y_\varepsilon
\end{equation}
satisfies the following three properties.
\begin{enumerate}[(a)]
    \item There is a family of functions \begin{equation}\label{e:fepsilon}
        f_\varepsilon\colon \Sigma\to(0,\infty),\qquad f_\varepsilon(q,v) =(\det A_q)^{\frac{1}{2k}}+O(\varepsilon).
    \end{equation} such that $f_\varepsilon \tilde X_\varepsilon$ is conjugated to $X_H|_{H^{-1}(\frac12\varepsilon^2)}$;
    \item For every $(q,v)\in\Sigma$, we have $\tilde Y_\varepsilon(q,v)\in\ker\tilde\tau=\ker\tau$ and
    \begin{equation}\label{e:tildeY}
    \tilde Y_\varepsilon(q,v)=\begin{pmatrix}
        \varepsilon^2(\det A)^{-\frac{1}{2k}}\zeta_q(v,v)+O(\varepsilon^3)\\
O(\varepsilon)    \end{pmatrix}=\begin{pmatrix}
    O(\varepsilon^2)\\ O(\varepsilon)\end{pmatrix};
        \end{equation}
    \item The vector field $\tilde X_\varepsilon$ generates the kernel of a two-form $\omega_\varepsilon$ on $\Sigma$ which is closed, cohomologous to $\pi^*\sigma$, and has the expression
\begin{equation}
\omega_\varepsilon=
\begin{pmatrix}
        \sigma+O(\varepsilon) &O(\varepsilon^2) \\
        O(\varepsilon^2) & \varepsilon^2\rho+O(\varepsilon^3)
    \end{pmatrix}
\end{equation}
\end{enumerate}
\end{lemma}
\begin{proof}
The vector field $I_\varepsilon^*F^* (\tilde X_H|_{H^{-1}(\frac12\varepsilon^2)})$ is tangent to $\Sigma$ and has the expression
\begin{equation}\label{e:IF}
I_\varepsilon^*F^* \big(\tilde X_H|_{H^{-1}(\frac12\varepsilon^2)}\big)=\tilde X_0+\begin{pmatrix}
\varepsilon^2(\det A)^{-\frac{1}{2k}}\zeta_q(v,v)+O(\varepsilon^3)\\O(\varepsilon)    
\end{pmatrix}.
\end{equation}
thanks to \eqref{e:F^*XHtilde} and \eqref{e:scaling}. Moreover, it satisfies
\begin{equation}\label{e:gepsilon}
g_\varepsilon:=\tilde\tau(I_\varepsilon^*F^*(\tilde X_H|_{H^{-1}(\frac12\varepsilon^2)}))=\tilde\tau( \tilde X_0)+\tau(O(\varepsilon))=1+O(\varepsilon).
\end{equation}
Therefore,
\begin{equation}\label{e:gepsilon2}
g_\varepsilon^{-1}=1+O(\varepsilon)
\end{equation}
and we define
\begin{equation}\label{e:tildeY2}
\tilde Y_\varepsilon:=g_\varepsilon^{-1} I_\varepsilon^*F^*\big(\tilde X_H|_{H^{-1}(\frac12\varepsilon^2)}\big)-\tilde X_0,
\end{equation}
so that $\tilde\tau(\tilde Y_\varepsilon)=0$ and
\begin{equation}
\tilde X_0+\tilde Y_\varepsilon=f_\varepsilon^{-1}I_\varepsilon^*F^*\big(X_H|_{H^{-1}(\frac12\varepsilon^2)}\big),\qquad f_\varepsilon:=(\det A)^{\frac{1}{2k}}g_\varepsilon.  
\end{equation}
Therefore, $f_\varepsilon\tilde X_\varepsilon$ is conjugate to $X_H|_{H^{-1}(\frac12\varepsilon^2)}$. From \eqref{e:gepsilon}, we see that $f_\varepsilon$ satisfies \eqref{e:fepsilon}. From \eqref{e:tildeY2}, \eqref{e:gepsilon2} and \eqref{e:IF}, we see that $\tilde Y_\varepsilon$ satisfies \eqref{e:tildeY}. This shows (a) and (b).

To prove (c), we define
\begin{equation}
\omega_\varepsilon:=I^*_\varepsilon F^*\big(\omega|_{H^{-1}(\frac12\varepsilon^2)}\big).    
\end{equation}
By Lemma \ref{order-of_F_lemma}, we know that $F^*(\omega|_{H^{-1}(\frac12\varepsilon^2)})$ is closed and cohomologous to $\pi^*\sigma$. Hence, $\omega_\varepsilon$ is closed and cohomologous to $I^*_\varepsilon\pi^*\sigma=\pi^*\sigma$.
Since $\tilde X_\varepsilon$ is, up to multiplication, conjugate to $X_H|_{H^{-1}(\frac12\varepsilon^2)}$, we see that the kernel of $\omega_\varepsilon$ is generated by $\tilde X_\varepsilon$. Using the formula for $I_\varepsilon$ and $d I_\varepsilon$, and formula \eqref{e:Fomega} for $F^*\omega$, we obtain the desired formula for $\omega_\varepsilon$.
\end{proof}
 
\section{Theorem \ref{thmA}, Step 1: Existence of a Limit Period}\label{sec:limit-period}

Suppose now that the Hamiltonian $H$ is Zoll along a sequence of energies $\tfrac12\varepsilon_n^2$ converging to zero. This means that the corresponding
sequence of rescaled vector fields $\tilde X_{\varepsilon_n}$ induces a free circle action on $\Sigma$ up to a smooth time reparametrization. Concretely, this means that there is a sequence of continuous functions
\begin{equation}
T_n\colon \Sigma\to(0,\infty)
\end{equation}
assigning to each point on $\Sigma$ the minimal period of the orbit of $\tilde X_{\varepsilon_n}$ passing through it. The key result of this section is the following one.

\begin{theorem}\label{t:T_n}
If $H$ is Zoll along a sequence of energies $\tfrac12\varepsilon_n^2$ converging to zero, then
\begin{enumerate}
\item $\min T_n$ is uniformly bounded away from zero,
\item $\max T_n$ is uniformly bounded from above,
\item $(\max T_n-\min T_n)$ converges to zero. 
\end{enumerate}
\end{theorem}
\begin{proof}
Property (1) follows from the fact that the limit vector field $\tilde X_0$ is nowhere vanishing \cite[Chapter 4, Proposition 1]{Hofer_Zehnder_book}. 

Let us prove Property (2) by contradiction. Assume that, up to taking a subsequence, $\max T_n\to\infty$. This means that there is a sequence $z_n$ of periodic orbits of $\tilde X_{\varepsilon_n}$ such that the corresponding sequence of minimal periods $\tilde T_n$ satisfies $\tilde T_n\to\infty$.

By the work of Ginzburg--Gürel \cite{Ginzburg_gurel} and Usher \cite{usher}, for every number $\varepsilon>0$ there exists a periodic orbit of $X_H$ on $\Sigma_\varepsilon$ with period bounded above uniformly in $\varepsilon$. Indeed, the existence of such a periodic orbit is proven in \cite[Theorem 1.1]{Ginzburg_gurel} if the first Chern class $c_1(TM)$ satisfies: $c_1(TM)\neq [0] \in H^2(M,\mathbb{R})$ or $c_1(TM) =[0] \in H^2(M,\mathbb{R})$ and the symplectic manifold $(Q,\sigma)$ is spherically rational. We recall that the the symplectic manifold $(Q,\sigma)$ is called spherically rational if the set of periods of $\sigma$ over spheres in $Q$ is discrete or equivalently, the following holds (c.f \cite[pp. 866]{Ginzburg_gurel}) :
\[
\lambda_0:=\inf \{\vert \langle\sigma,u \rangle\vert\;\vert \; u\in \pi_2(Q), \langle \sigma,u\rangle \neq 0 \}>0
\]
This assumption of spherical rationality of $Q$ when $c_1(TM)=[0] \in H^2(M,\mathbb{R})$ is removed in \cite[Theorem 1.4]{usher}. Since $f_\varepsilon=(\det A_q)^{\frac{1}{2k}}+O(\varepsilon)$ and $f\varepsilon\tilde X_\varepsilon$ is conjugated to $X_H|_{H^{-1}(\frac{1}{2}\varepsilon^2)}$ by Lemma \ref{l:tildeX}.(a), we also conclude that for all $\varepsilon>0$ there exists a periodic orbit of $\tilde X_\varepsilon$ with period bounded above uniformly in $\varepsilon$. Thus there is a positive real number $\tilde S$ and a sequence $w_n$ of periodic orbits of $\tilde X_{\varepsilon_n}$ such that the corresponding sequence of minimal periods $\tilde S_n$ satisfies $\tilde S_n\leq \tilde S$.  Up to taking a further subsequence, there is an interval $I$ independent of $n$ which has positive length and is contained in the interval between $\tilde S_n$ and $\tilde T_n$ for all $n$. Since $\Sigma$ is connected and the function $T_n$ is continuous, by the Intermediate Value Theorem, we deduce that for every $T\in I$, there is a periodic orbit $y_n$ of period $T$ of $\tilde X_{\varepsilon_n}$. Passing to the limit and possibly taking a subsequence, the continuous dependence of solutions to ordinary differential equations on the initial condition tells us that $y_n$ uniformly converges to a periodic orbit $y\colon\mathbb R\to \Sigma$ for $\tilde X_0$ with (possibly not minimal) period $T\in I$. Choosing $T$ to be outside the countable set $2\pi\cdot(\mathbb Q_+^k)^\frac{1}{k}$ (which is possible since $I$ is an interval of positive length), we deduce that $y$ is contained in some $\Sigma_q$ where not all orbits of $\tilde X_0$ are periodic.

Choose a number $S\geq T$ such that no orbit of $\tilde X_0$ on $\Sigma_q$ has a period (minimal or not minimal) equal to $S$. This is possible since the set of periods belongs to the countable set
\begin{equation}
\Big\{\frac{2\pi m}{\tilde a_j(q)}\ \Big|\ m\in\mathbb N,\ j=1,\ldots,k \Big\}.
\end{equation}
Fix $v\in \Sigma_q$ with the property that the orbit of $\tilde X_0$ with initial condition $(q,v)$ is not periodic. Let $x_n$ be the periodic orbit of $\tilde X_{\varepsilon_n}$ such that $x_n(0)=(q,v)$ and denote by $\tilde R_n$ its minimal period. From the choice of $v$, we deduce that $\lim_{n\to\infty}\tilde R_n=\infty$ and therefore we can assume up to taking a subsequence that $\tilde R_n\geq S$ for all $n$. Let us define $(q_n,v_n):=y_n(0)$, where $y_n$ is the orbit of $\tilde X_{\varepsilon_n}$ defined above. Thus the distance $d_n:=d(q_n,q)$ is converging to $0$. Let $B_{d_n}(q)$ be the closed ball in the manifold $Q$ with center $q$ and radius $d_n$. Since $x_n(0)=(q,v)$ and $y_n(0)=(q_n,v_n)$ both belong to the connected set $\Sigma|_{B_{d_n}(q)}$ and $T_n(q_n,v_n)=T\leq S\leq R_n=T_n(q,v)$, we again deduce from the Intermediate Value Theorem that there exists $(\tilde q_n,\tilde v_n)\in \Sigma|_{B_{d_n}(q)}$ such that $T_n(\tilde q_n,\tilde v_n)=S$ for all $n$. Up to taking a subsequence, $(\tilde q_n,\tilde v_n)\to (q,\tilde v)$ and the orbit of $\tilde X_0$ with initial condition $(q,\tilde v)$ has period $S$, contradicting the choice of $S$.

The proof of Property (3) is very similar to Property (2). Indeed, assume by contradiction that Property (3) does not hold. Thus, up to taking a subsequence, there exists $\delta>0$ such that $\max T_n-\min T_n\geq \delta$ for all $n$. This means that there are periodic orbits $w_n$ and $z_n$ of $\tilde X_{\varepsilon_n}$ having periods $\tilde S_n$ and $\tilde T_n$ such that $\tilde T_n-\tilde S_n\geq \delta$. By Property (1) and (2), up to taking a subsequence, we can suppose that $\tilde S_n\to S>0$ and $\tilde T_n\to \tilde T$. Therefore, upon taking a further subsequence, there is an interval $I$ independent of $n$ which has positive length and is contained in the interval between $\tilde S_n$ and $\tilde T_n$ for all $n$. The rest of the argument is as in the proof of Property (2).
\end{proof}
\begin{corollary}\label{c:T}
Assume that $H$ is Zoll along a sequence of energies $\tfrac12\varepsilon^2_n$ converging to a closed, connected, symplectic Morse--Bott minimum and let $T_n\colon \Sigma\to(0,\infty)$ be the corresponding sequence of period functions of $\tilde X_{\varepsilon_n}$. Then there is a positive number $T>0$ and a subsequence $n_\ell$ such that $T_{n_\ell}$ converges uniformly to $T$. As a consequence, 
\begin{enumerate}
    \item the flow of $\tilde X_0$ is $T$-periodic, that is, $\Phi_{\tilde X_0}^T=\mathrm{id}_\Sigma$;
    \item the eigenvalues $\tilde a_1,\ldots,\tilde a_k\colon Q\to(0,\infty)$ and, for all $\tilde a>0$, the multiplicities $k_{\tilde a}\colon Q\to\mathbb N$ are constant functions on $Q$.
\end{enumerate}
\end{corollary}
\begin{proof}
The existence of the subsequence $n_\ell$ and of the positive number $T>0$ follows by Theorem \ref{t:T_n}. In particular, for all $(q,v)\in\Sigma$, $\Phi_{\tilde X_{\varepsilon_n}}^{T_n(q,v)}(q,v)=(q,v)$. By the continuous dependence of the solutions to ordinary differential equations on the initial conditions, it follows that $\Phi_{\tilde X_0}^T(q,v)=(q,v)$. Finally, by the explicit formula for $\Phi_{\tilde X_0}^T(q,v)$, we see that there exist natural numbers $n_1(q),\ldots,n_k(q)$ such that 
\begin{equation}
T\tilde a_j(q)=2\pi n_j(q),\qquad \forall\,j=1,\ldots,k.
\end{equation}
Since the functions $\tilde a_j$ are continuous, it follows that the functions $n_j$ are continuous and hence constant on the connected manifold $Q$. Thus, also the functions $\tilde a_j$ are constant on $Q$. Since for every $q\in Q$ the endomorphisms $\tilde A_q$ are diagonalizable and their eigenvalues are constant in $q$, it follows that the multiplicities of the eigenvalues are constant on $Q$, as well.
\end{proof}
\section{Theorem \ref{thmA}, Step 2: Bifurcation Implies Conformality}\label{sec:bifurcation-conformality}
\subsection{The Submanifold of Orbits with Minimal Period}
Recall that
\begin{equation}
\Phi^t_{\tilde X_0}(q,v)=(q,e^{t\tilde A_q}v),\qquad\forall\,t\in\mathbb R,\ \forall\,(q,v)\in E,
\end{equation}
where $\tilde A_q:=(\det A_q)^{-\frac{1}{2k}}A_q$. We have shown in the previous section that there is a number $T>0$ such that $\Phi^T_{\tilde X_0}=\mathrm{id}_\Sigma$ and all the spectral numbers $\tilde a_1\geq \ldots\geq \tilde a_k$ and the multiplicities $k_{\tilde a}$ for all $\tilde a>0$ are constant on $Q$. Let us start by drawing two useful conclusions from these facts.
\begin{lemma}\label{l:connection}
The flow $\Phi_{\tilde X_0}$ preserves $\rho$ and $\gamma$. Moreover, there exists a connection $\nabla$ on $E$ with the following two properties.
\begin{enumerate}[(a)]
    \item The connection $\nabla$ preserves $\tilde A$ and $\gamma$. In particular, 
\begin{equation}\label{e:dphi}
d_{(q,v)}\Phi^t_{\tilde X_0}=\begin{pmatrix}
1&0\\
0& e^{t\tilde A_q}
\end{pmatrix},\qquad  \mathcal H|_\Sigma\subset T\Sigma,
\end{equation}
where the block decomposition is with respect to the horizontal-vertical splitting.
    \item The flow $\Phi_{\tilde X_0}$ preserves $\tau$ and $\tilde \tau$.
\end{enumerate}
\end{lemma}
\begin{proof}
First, recall that we are considering a connection $\nabla$ on $E$ preserving $\gamma$. To show that $\Phi_{\tilde X_0}$ preserves $\rho$ and $\gamma$, it is enough to prove that $A_q$ is symmetric with respect to $\rho$ and antisymmetric with respect to $\gamma$. Both properties follow from \eqref{def_of_A}, which implies
\begin{equation}
\rho_q(w,A_qv)=\gamma_q(w,v),\qquad \gamma_q(A_qw,v)=\rho_q(A_qw,A_qv).
\end{equation}
Let $\nabla'$ be a connection which preserves $\gamma$. Since $\Phi_{\tilde X_0}$ generates an action of period $T$, we can average the connection $\nabla'$ by
\begin{equation}
\nabla:=\frac{1}{T}\int_0^T (\Phi^t_{\tilde X_0})^*\nabla' dt.
\end{equation}
The averaged connection is invariant under $\Phi_{\tilde X_0}$ and hence $\nabla$ preserves $\tilde A$.
Moreover, since $\gamma$ was $\nabla'$-parallel and $\Phi_{\tilde X_0}$ preserves $\gamma$, we conclude that $\gamma$ is also $\nabla$-parallel. Finally, using \eqref{e:dphi}, we get
\begin{equation}
\big((\Phi_{\tilde X_0}^t)^*\tau\big)_{(q,v)}(\xi)=\rho_q(e^{t\tilde A_q}v,e^{t\tilde A_q}\xi^\nabla)=\rho_q(v,\xi^\nabla)    
\end{equation}
and similarly for $\tilde \tau$.
\end{proof}
In the second preliminary result, we describe the set of orbits $\Sigma_{\min}$ of $\tilde X_0$ with minimal period more closely. Here, we define 
\begin{equation}
    \Sigma_{\min}:=\big\{(q,v)\in\Sigma\ \big|\ \Phi_{\tilde X_0}^{T_{\min}}(q,v)=(q,v)\big\},
\end{equation}
and
\begin{equation}
T_{\min}:=\frac{2\pi}{\tilde a_1}
\end{equation}
is the minimal period of a periodic orbit of $\Phi_{\tilde X_0}$.
\begin{lemma}\label{l:minsub}
The set $\Sigma_{\min}$ is an embedded submanifold of $\Sigma$ and the restriction of the projection $\pi\colon \Sigma_{\min}\to Q$ is a smooth $S^{2k_{\tilde a_1}-1}$ bundle. Moreover, \begin{enumerate}[(a)]
    \item $\Sigma_{\min}$ is non-degenerate for the flow of $\tilde X_0$, that is,
\begin{equation}\label{e:sigmanondeg}
\ker \Big(d_{(q,v)}\Phi_{\tilde X_0}^{T_{\min}}-1_{T_{(q,v)}\Sigma}\Big)=T_{(q,v)}\Sigma_{\min},\qquad \forall\,(q,v)\in\Sigma_{\min};
\end{equation}
\item there are $\Phi_{\tilde X_0}$-invariant splittings
\begin{equation}\label{e:Tsigmamin}
\begin{aligned}
T_{(q,v)}\Sigma_{\min}&=\mathcal H_{(q,v)}\oplus T_v\Sigma_{\min,q},\\ 
T_{(q,v)}\Sigma|_{\Sigma_{\min}}&=T_{(q,v)}\Sigma_{\min}\oplus (T_v\Sigma_{\min,q})^\perp,
\end{aligned}    
\end{equation}
where $(T_v\Sigma_{\min,q})^\perp$ denotes the $\gamma_q$-orthogonal of $T_v\Sigma_{\min,q}$ inside $T_v\Sigma_q$.
\end{enumerate}
\end{lemma}
\begin{proof}
If $(q,v)\in\Sigma_{\min}$, then by \eqref{e:dphi}
\begin{equation}\label{e:kerdphi}
d_{(q,v)}\Phi^{T_{\min}}_{\tilde X_0}-1_{T_{(q,v)}\Sigma}=\begin{pmatrix}
0&0\\
0& e^{T_{\min}\tilde A_q}-1
\end{pmatrix}
\end{equation}
has constant rank equal to $2(k-k_{\tilde a_1})$. By the constant rank theorem, $\Sigma_{\min}$ is an embedded submanifold and \eqref{e:sigmanondeg} holds. Moreover, computing explicitly the kernel using \eqref{e:kerdphi}, we deduce the splittings \eqref{e:Tsigmamin}. These splittings are invariant under $\Phi_{\tilde X_0}$ since $\mathcal H$, $\gamma$ and $\Sigma_{\min}$ are invariant under $\Phi_{\tilde X_0}$. Finally, the first splitting implies that $\pi\colon \Sigma_{\min}\to Q$ is a submersion. The formula for $\Phi_{\tilde X_0}$ shows that the fibers of $\pi$ are spheres of dimension $2k_{\tilde a_1}-1$. 
\end{proof}
We now aim to prove that $\Sigma_{\min}=\Sigma$ so that $\tilde X_0$ is actually Zoll. As a byproduct, we will see that the whole sequence of functions $T_n$ (and not just a subsequence) uniformly converges to a constant. To this purpose, we need to upgrade the existence result of Ginzburg--Gürel \cite{Ginzburg_gurel} and Usher \cite{usher} to the following bifurcation theorem in the spirit of Kerman \cite{Kerman}.
\begin{theorem}\label{bifurcation_theorem_morse_bott}
Assume that there exists $T>0$ such that $\Phi_{\tilde X_0}^T=\mathrm{id}_\Sigma$. Let $\Sigma_{\min}$ be the manifold of orbits of minimal period $T_{\min}$. For every $\delta_0>0$ there is $\varepsilon_0>0$ such that if $\varepsilon\in(0,\varepsilon_0)$, then there are at least $\mathrm{Cup Length}(Q)+k_{\tilde a_1}$ periodic orbits of $\tilde X_\varepsilon$ with period $\tilde T_\varepsilon$ with the property that
\begin{equation}
|\tilde T_\varepsilon-T_{\min}|<\delta_0.
\end{equation}
\end{theorem}
Before going into the proof of the theorem, we state the promised corollary.
\begin{corollary}\label{cor: bifurcation}
Let $H$ be a Hamiltonian having a closed, connected, symplectic Morse-Bott minimum at $H=0$. If $H$ is Zoll along a sequence of energies $\tfrac12\varepsilon_n^2\to 0$, then
\begin{enumerate}
\item the flow of $\tilde X_0$ is Zoll with period $T_{\min}=2\pi$ and $\tilde A=J$ is an almost complex structure $J$ (compatible with $\rho$);
\item if $T_n\colon\Sigma\to(0,\infty)$ denotes the sequence of functions that gives the period of orbits of $\tilde X_{\varepsilon_n}$, then $T_n$ converges uniformly to $2\pi$. 
\end{enumerate}
\end{corollary}
\begin{proof}
Assume by contradiction that $\Phi_{\tilde X_0}$ is not Zoll. Thus $T_{\min}<T$, where $T$ is the positive number given by Corollary \ref{c:T}. Fix an $S\in (T_{\min},T)$. Theorem \ref{bifurcation_theorem_morse_bott} shows that for all $\varepsilon>0$ small enough $\tilde X_\varepsilon$ has a periodic orbit of (minimal) period $\tilde T_\varepsilon\leq S$. This contradicts the fact that the sequence of functions $T_{n_\ell}$ converges uniformly to $T>S$. This shows that $\Phi_{\tilde X_0}$ is Zoll which also implies that all the spectral numbers $\tilde a_1,\ldots,\tilde a_k$ are equal. Since $1=\det\tilde A_q=a_1^2\cdot \ldots\cdot a_k^2$, we deduce that all spectral numbers are equal to $1$. Hence, $T_{\min}=2\pi$ and $\tilde A$ is an almost complex structure. This finishes the proof of Property (1). Applying the above argument to any subsequence $n_{\ell'}$ of $n$, we find a further subsequence $n_{\ell'_\ell}$ such that $T_{\ell'_\ell}$ converges uniformly to $2\pi$. Hence, the whole sequence $T_n$ converges to $2\pi$.
\end{proof}
\subsection{Vector Fields along $\Sigma_{\min}$}
To prove Theorem \ref{bifurcation_theorem_morse_bott}, we need to study the space $\mathfrak{X}(\Sigma|_{\Sigma_{\min}})$ of vector fields of $\Sigma$ along the submanifold $\Sigma_{\min}$. Let $U\in\mathfrak{X}(\Sigma|_{\Sigma_{\min}})$. Equation \ref{e:Tsigmamin} yields a $\Phi_{\tilde X_0}$-invariant splitting
\begin{equation}\label{e:Tsigmasplit}
T_{(q,v)}\Sigma|_{\Sigma_{\min}}=T_{(q,v)}\Sigma_{\min}\oplus (T_v\Sigma_{\min,q})^\perp
\end{equation}
and we use it to decompose
\begin{equation}
U=U^\top+U^\perp.
\end{equation}
We define the average vector field
\begin{equation}
\overline U:=\frac{1}{T}\int_0^{T}d\Phi_{\tilde X_0}^{-t}\cdot  U(\Phi_{\tilde X_0}^t)dt.
\end{equation}
We denote by $\mathfrak X_0(\Sigma|_{\Sigma_{\min}})$ the space of vector fields with zero average. This space decomposes into a tangential and a vertical part according to the splitting \eqref{e:Tsigmasplit}:
\begin{equation}\label{e:split1}
\mathfrak X_0(\Sigma|_{\Sigma_{\min}})=\mathfrak X_0(\Sigma_{\min})\oplus \mathfrak X_0^\perp(\Sigma|_{\Sigma_{\min}}).
\end{equation}
For later purposes, we also define
\begin{equation}
\mathfrak X(\ker\tilde\tau|_{\Sigma_{\min}}):=\{W\in \mathfrak X(\Sigma_{\min}\ |\ \tilde\tau(W)=0\}
\end{equation}
and 
\begin{equation}
\mathfrak X_1(\Sigma|_{\Sigma_{\min}}):=\big\{U\in\mathfrak X(\Sigma|_{\Sigma_{\min}})\ \big|\ \overline{U^\top}=0 \big\},
\end{equation}
which in the splitting \eqref{e:Tsigmasplit} decomposes into
\begin{equation}\label{e:split2}
\mathfrak X_1(\Sigma|_{\Sigma_{\min}})=\mathfrak X_0(\Sigma_{\min})\oplus\mathfrak X^\perp(\Sigma|_{\Sigma_{\min}}).
\end{equation}
Finally, let $\bar{\mathfrak X}(\Sigma|_{\Sigma_{\min}})\subset \mathfrak X(\Sigma|_{\Sigma_{\min}})$ be the space of vector fields that are invariant under the flow of $\tilde X_0$. The main properties of the spaces of vector fields introduced above are contained in the following lemma. To formulate it, we notice that since $\Phi_{\tilde X_0}$ induces a free $\mathbb R/T_{\min}\mathbb Z$-action on $\Sigma_{\min}$, the quotient map $\bar\pi\colon \Sigma_{\min}\to R$ is a circle bundle over a closed manifold $R$. There is an isomorphism
\begin{equation}
\mathfrak X(R)\to \bar C^\infty(\Sigma_{\min}),\qquad \bar\nu\mapsto \nu:=\bar\nu\circ\bar\pi
\end{equation}
between the space of functions on $\Sigma_{\min}$ which are invariant under $\Phi_{\tilde X_0}$ and the space of functions on $R$.
\begin{lemma}\label{l:vectors}
We have isomorphisms
\begin{equation}
\bar{\mathfrak X}(\Sigma|_{\Sigma_{\min}})\cong(\bar C^\infty(\Sigma_{\min})\tilde X_0)\oplus \bar{\mathfrak X}(\ker\tilde\tau|_{\Sigma_{\min}})\cong C^\infty(R)\oplus \mathfrak X(R)
\end{equation}
given by
\begin{equation}
U\mapsto \nu\tilde X_0+W\mapsto (\bar\nu,\bar W).
\end{equation}
Moreover, the average map
\begin{equation}
\mathfrak X(\Sigma|_{\Sigma_{\min}})\to \bar{\mathfrak X}(\Sigma|_{\Sigma_{\min}}),\qquad U\mapsto \overline U
\end{equation}
is a linear projection which is continuous from the $C^k$-topology to the $C^k$-topology for any $k\geq 0$. Therefore, we have a splitting
\begin{equation}\label{e:splittings}
\mathfrak X(\Sigma|_{\Sigma_{\min}})\cong\bar{\mathfrak X}(\Sigma|_{\Sigma_{\min}})\oplus \mathfrak X_0(\Sigma|_{\Sigma_{\min}}).
\end{equation}
Finally, there is an isomorphism
\begin{equation}
\mathcal L\colon \mathfrak X_1(\Sigma|_{\Sigma_{\min}})\to \mathfrak X_0(\Sigma|_{\Sigma_{\min}}),\qquad U\mapsto [\tilde X_0,U],
\end{equation}
where $\mathcal L^{-1}$ is continuous from the $C^k$-topology to the $C^k$-topology. Furthermore, $\mathcal L$ and $\mathcal L^{-1}$ preserve the splittings \eqref{e:split1} and \eqref{e:split2}, and the horizontal-vertical splitting given by $\nabla$.
\end{lemma}
\begin{proof}
Let $U\in\bar{\mathfrak X}(\Sigma|_{\Sigma_{\min}})$. In particular, $U=d\Phi_{\tilde X_0}^{T_{\min}}U$, and \eqref{e:sigmanondeg} implies $U\in\mathfrak X(\Sigma_{\min})$. The function $\tilde\tau(U)$ is invariant under $\Phi_{\tilde X_0}$ since $\tilde\tau$ is invariant by Lemma \ref{l:connection}.(b). Therefore, it can be written as $\tilde\tau(U)=f\circ\bar\pi$ for a function $f\in C^\infty(R)$. All invariant vector fields parallel to $\tilde X_0$ can be written as $f\circ\bar\pi \tilde X_0$. The vector field $W:=U-f\circ\tilde X_0$ is also invariant and in the kernel of $\tilde \tau$ since $\tilde\tau(\tilde X_0)=1$. The map $d_z\bar\pi\colon\ker\tilde\tau_z|_{\Sigma_{\min}}\to T_{\bar\pi(z)}R$ is an isomorphism. If $W$ is invariant under the flow of $\Phi_{\tilde X_0}$, then $\bar W_{\bar\pi(z)}:=d_z\bar\pi W$ depends only on $\bar\pi(z)$ and hence determines a vector field on $R$. Vice versa, given a vector field on $R$ we can lift it to a vector field in $\ker\tilde\tau$ that is invariant under $\Phi_{\tilde X_0}$.

By the periodicity of the flow and the change-of-variable formula in integrals, the average map takes values in $\bar{\mathfrak X}(\Sigma|_{\Sigma_{\min}})$. If $U$ is already invariant, then $d\Phi_{\tilde X_0}^{-t} U(\Phi_{\tilde X_0}^t)$ for all $t$ and therefore $\overline U=U$. Hence the average map is a projection onto $\bar{\mathfrak X}(\Sigma|_{\Sigma_{\min}})$ and from the formula we see that it is continuous from the $C^k$-topology to the $C^k$-topology. Since we can decompose a space as the direct sum of the range and the kernel of a projection, the splittings in \eqref{e:splittings} follow.

Let us now study the map $\mathcal L$. First, we observe that
\begin{equation}
\overline{[\tilde X_0,U]}=[\tilde X_0,\overline{U}]=0,
\end{equation}
where in the last equality we used that $\overline{U}$ is invariant under $\tilde X_0$. Therefore, $\mathcal L$ takes values in $\mathfrak X_0(\Sigma|_{\Sigma_{\min}})$. Thus, let $Z\in\mathfrak X_0(\Sigma|_{\Sigma_{\min}})$ and consider the equation 
\begin{equation}\label{e:bracket}
[\tilde X_0,U]=Z.
\end{equation}
This equation is equivalent to
\begin{equation}\label{e:integral1}
d\Phi_{\tilde X_0}^{-s}\cdot U(\Phi_{\tilde X_0}^s)=U+\int_0^sd\Phi_{\tilde X_0}^{-t}\cdot Z(\Phi_{\tilde X_0}^t)dt.
\end{equation}
We write $U=U^\top+U^\perp$ and $Z=Z^\top+Z^\perp$ using the splitting \eqref{e:Tsigmasplit}. Equations \eqref{e:bracket} and \eqref{e:integral1} then decouple into an equation for $U^\top$ and $Z^\top$, and an equation for $U^\perp$ and $Z^\perp$.

For the tangential part, we take the average of \eqref{e:integral1} for $s\in[0,T_{\min}]$ and recall that $\overline{U^\top}=0$ to obtain
\begin{equation}\label{e:tan}
U^\top=\frac{1}{T_{\min}}\int_0^{T_{\min}}t\,d\Phi_{\tilde X_0}^{-t}\cdot Z^\top(\Phi_{\tilde X_0}^t)dt,
\end{equation}
where we used that $\overline{Z^\top}=0$ and that we can take the average on $[0,T_{\min}]$ instead of taking it on $[0,T]$ since $\Phi_{\tilde X_0}$ is $T_{\min}$-periodic on $\Sigma_{\min}$. 

For the orthogonal part, we evaluate \eqref{e:integral1} for $s=T_{\min}$ and use that $\Phi_{\tilde X_0}^{T_{\min}}=\mathrm{Id}$ on $\Sigma_{\min}$ and that $1-d\Phi_{\tilde X_0}^{T_{\min}}$ is invertible on the orthogonal part since $\Sigma_{\min}$ is not degenerate by Lemma \ref{l:minsub}.(a). We obtain
\begin{equation}\label{e:nor}
U^\perp=(1-d\Phi_{\tilde X_0}^{T_{\min}})^{-1}d\Phi_{\tilde X_0}^{T_{\min}}\int_0^{T_{\min}}d\Phi_{\tilde X_0}^{-t}\cdot Z^\perp(\Phi_{\tilde X_0}^t)dt.
\end{equation}
From \eqref{e:tan} and \eqref{e:nor}, we see that $\mathcal L^{-1}$ is continuous from the $C^k$-topology to the $C^k$-topology, and that $\mathcal L$ and $\mathcal L^{-1}$ preserve the horizontal-vertical splitting of $\nabla$ since so does $d\Phi_{\tilde X_0}$ by Lemma \ref{l:connection}.(a).
\end{proof}
\subsection{The Bottkol Normal Form Theorem}
The first step in the proof of Theorem \ref{bifurcation_theorem_morse_bott} is an enhancement of a normal form theorem of Bottkol on the bifurcation of periodic orbits from a non-degenerate periodic submanifold \cite{Bottkol}. To formulate the statement, we need a couple of definitions.   Let $\nabla^Q$ be any torsion-free affine connection on $TQ$ and let $\nabla$ be the affine connection given by Lemma \ref{l:connection}. We define a connection $\tilde \nabla^E$ on $TE$ by
\begin{equation}\label{e:connection_E}
\tilde \nabla^E:=\pi^*\nabla^Q\oplus \pi^*\nabla.
\end{equation}
and a connection $\tilde \nabla$ on $T\Sigma$ by
\begin{equation}\label{e:connection}
 \tilde\nabla:= \text{fiberwise $\gamma$-orthogonal projection of $\tilde\nabla^E$ from $TE|_{\Sigma}$ to $T\Sigma$}.  
\end{equation}
For every $\xi\in T\Sigma$ we write $\xi^{\tilde v}\in T(T\Sigma)$ for the vertical lift of $\xi$ with respect to the splitting of $T(T\Sigma)$ induced by $\tilde\nabla$. Let $\exp\colon T\Sigma\to \Sigma$ be the exponential map of $\tilde\nabla$.

For every $U\in \mathfrak X(\Sigma|_{\Sigma_{\min}})$, we define the map
\begin{equation}\label{e:u}
u\colon \Sigma_{\min}\to\Sigma,\qquad u(z):=\exp(z,U(z)),\qquad\forall\,z\in\Sigma_{\min}.
\end{equation}
For every $z\in\Sigma_{\min}$, we also define the $\tilde\nabla$-geodesic
\begin{equation}\label{e:delta}
\delta_{U(z)}\colon [0,1]\to\Sigma,\qquad \delta_{U(z)}(t):=\exp(z,tU(z)),
\end{equation}
and consider the Jacobi field endomorphism
\begin{equation}\label{e:P}
\mathcal P_U(z)\colon T_z\Sigma\to T_{u(z)}\Sigma,\quad \mathcal P_U(z)\cdot \xi:=d_{(z,U(z))}\exp{} \cdot \xi^{\tilde v}
\end{equation}
and the parallel transport endomorphism
\begin{equation}\label{e:Q}
\mathcal Q_U(z)\colon T_z\Sigma\to T_{u(z)}\Sigma, \quad\mathcal Q_U(z):=\text{`parallel transport along $\delta_{U(z)}$'}.
\end{equation}
We denote by $\mathcal C$ the class of paths of homomorphisms
\begin{equation}
\Lambda_\varepsilon\colon T\Sigma_{\min}\to T\Sigma|_{\Sigma_{\min}},\qquad \forall\,\varepsilon>0 \text{ small enough.}
\end{equation}
We define the subclass $\mathcal D\subset\mathcal C$ of paths such that
\begin{equation}\label{e:mathcald}
    \mathcal D:=\left\{\Lambda\in\mathcal C\ \Big|\ \Lambda_\varepsilon=\begin{pmatrix}
        O(\varepsilon)&O(\varepsilon^2)\\
        O(\varepsilon)&O(\varepsilon)
    \end{pmatrix}\right\},
\end{equation}
where the block decomposition is with respect to the horizontal-vertical splitting of $\nabla$.
\begin{theorem}\label{bottkol_sec_4}
Let
\begin{equation}
\tilde X_\varepsilon=\tilde X_0+\tilde Y_\varepsilon
\end{equation} be the path of vector fields on $\Sigma$ where $\tilde Y_\varepsilon$ satisfies \eqref{e:tildeY}. For every $\varepsilon>0$ small enough there exist unique vector fields and a function
\begin{equation}
U_\varepsilon \in\mathfrak X_1(\Sigma|_{\Sigma_{\min}}),\qquad V_\varepsilon\in\bar{\mathfrak X}(\ker\tilde\tau|_{\Sigma_{\min}}),\qquad \mu_{\varepsilon}\in \bar C^\infty(\Sigma_{\min})
\end{equation}
such that, defining \begin{equation}
u_\varepsilon:=\exp(U_\varepsilon),\qquad \nu_\varepsilon:=1+\mu_\varepsilon,
\end{equation}
we get
\begin{equation}\label{e:mainbott}
    \nu_{\varepsilon}\tilde X_{\varepsilon} \circ u_\varepsilon = du_\varepsilon\cdot\tilde X_0 -\mathcal{P}_{U_{\varepsilon}}\cdot V_{\varepsilon}
\end{equation}
and the following estimates hold
\begin{equation}\label{e:estimates}
\begin{aligned}
(a)&\ \ U_\varepsilon=O(\varepsilon),\qquad V_\varepsilon=O(\varepsilon),\qquad\mu_\varepsilon=O(\varepsilon),\\
(b)&\ \ U^\pi_\varepsilon=O(\varepsilon^2), \qquad\mathcal Q^{-1}_{U_\varepsilon}\cdot du_\varepsilon\in \iota+\mathcal D,\qquad \mathcal Q^{-1}_{U_\varepsilon}\cdot\mathcal P_{U_\varepsilon}\in \iota+\mathcal D,
\end{aligned}
\end{equation}
where $\iota\colon T\Sigma_{\min}\to T\Sigma|_{\Sigma_{\min}}$ is the inclusion and $\mathcal D$ is defined in \eqref{e:mathcald}.
\end{theorem}
We postpone the proof of Theorem \ref{bottkol_sec_4} to Section \ref{sec:proof-bifurcation}. In the next subsection, we prove Theorem \ref{bottkol_sec_4}. The proof will be based on Claim \eqref{e:claim}, which will be proved in Subsection \ref{sub:iso}.

\subsection{The Proof of the Bifurcation Theorem \ref{bifurcation_theorem_morse_bott}}\label{subsec:duality}
Consider the path of two-forms $\omega_\varepsilon$ on $\Sigma$ given by Lemma \ref{l:tildeX}.(c). Let $U_\varepsilon$, $V_\varepsilon$, $\nu_\varepsilon$ and $u_\varepsilon:=\exp(U_\varepsilon)$ be as in Theorem \ref{bottkol_sec_4}. By the fact that $u_\varepsilon$ is homotopic to the identity and $\omega_\varepsilon$ is cohomologous to $\pi^*\sigma$ by Lemma \ref{l:tildeX}, we know that there exists a one-form $\lambda_\varepsilon$ such that
\begin{equation}\label{e:uomega}
u_\varepsilon^*\omega_\varepsilon=d\lambda_\varepsilon+\pi^*\sigma.
\end{equation}
We define the reduced action functional
\begin{equation}
S_\varepsilon\colon \Sigma_{\min}\to\mathbb R,\qquad S_\varepsilon(z):=\int_{\mathbb R/T_{\min}\mathbb Z}\eta_z^*\lambda_\varepsilon,
\end{equation}
where $\eta_z\colon \mathbb R/T_{\min}\mathbb Z\to\Sigma_{\min}$ is the orbit $\eta_z(t):=\Phi^t_{\tilde X_0}(z)$ of $\tilde X_0$ through $z$. Notice that $S_\varepsilon\in\bar C^\infty(\Sigma_{\min})$ and therefore $S_\varepsilon=\bar S_\varepsilon\circ\bar\pi$ for some $\bar S_\varepsilon\in C^\infty(R)$. We define $\eta_z^\varepsilon:=u_\varepsilon\circ\eta_z$ and compute the differential of $S_\varepsilon$ as
\begin{equation*}
\begin{aligned}
d_zS_\varepsilon\cdot\xi&=\int_0^{T_{\min}}d\lambda_\varepsilon\Big(\tilde X_0(\eta_z(t)),d_z\Phi^t_{\tilde X_0}\cdot\xi\Big)dt\\
&=\int_0^{T_{\min}}u_\varepsilon^*\omega_\varepsilon\Big(\tilde X_0(\eta_z(t)),d_z\Phi^t_{\tilde X_0}\cdot\xi\Big)dt\\
&=\int_0^{T_{\min}}(\omega_\varepsilon)_{\eta^\varepsilon_z(t)}\Big(d_{\eta_z(t)}u_\varepsilon\cdot\tilde X_0(\eta_z(t)),d_{\eta_z(t)}u_\varepsilon \cdot d_z\Phi^t_{\tilde X_0}\cdot\xi\Big)dt\\
&=\int_0^{T_{\min}}(\omega_\varepsilon)_{\eta^\varepsilon_z(t)}\Big(\mathcal P_{U_\varepsilon}(\eta_z(t))\cdot V_\varepsilon(\eta_z(t)),d_{\eta_z(t)}u_\varepsilon \cdot d_z\Phi^t_{\tilde X_0}\cdot\xi\Big)dt\\
&=\int_0^{T_{\min}}(\omega_\varepsilon)_{\eta^\varepsilon_z(t)}\Big(\mathcal P_{U_\varepsilon}(\eta_z(t))\cdot d_z\Phi^t_{\tilde X_0}\cdot V_\varepsilon(z),d_{\eta_z(t)}u_\varepsilon \cdot d_z\Phi^t_{\tilde X_0}\cdot\xi\Big)dt,
\end{aligned}
\end{equation*}
where in the first step we used Cartan's magic formula, in the second step we used \eqref{e:uomega} together with the fact that $\tilde X_0\in\ker\pi^*\sigma$, in the third step we used the definition of pull-back, in the fourth step we used \eqref{e:mainbott} together with the fact that $\tilde X_\varepsilon\in \ker\omega$ by Lemma \ref{l:tildeX}.(c), and in the fifth step we used that $V_\varepsilon$ is invariant under the flow of $\tilde X_0$ by Theorem \ref{bottkol_sec_4}.
Therefore, if we define the duality homomorphism
$\mathcal K_\varepsilon\colon \ker\tilde\tau|_{\Sigma_{\min}}\to (\ker\tilde\tau|_{\Sigma_{\min}})^*$ as \begin{equation}
\mathcal K_\varepsilon(w)\cdot\xi=\int_0^{T_{\min}}(\omega_\varepsilon)_{\eta^\varepsilon_z(t)}\Big(\mathcal P_{U_\varepsilon}(\eta_z(t))\cdot d_z\Phi^t_{\tilde X_0}\cdot w,d_{\eta_z(t)}u_\varepsilon \cdot d_z\Phi^t_{\tilde X_0}\cdot\xi\Big)dt,
 \end{equation}
 we see that 
 \begin{equation}
d_zS_\varepsilon=\mathcal K_\varepsilon(V_\varepsilon(z)).
 \end{equation}
 We claim that
\begin{equation}\label{e:claim}
\text{$\mathcal K_\varepsilon$ is an isomorphism for $\varepsilon$ small enough}.      
\end{equation}
Given the claim, let us finish the proof of the theorem. Let $z\in \Sigma_{\min}$ be such that $V_\varepsilon(z)=0$. Since $V_\varepsilon$ is invariant, then $V_\varepsilon(\delta_z(t))=0$ for all $t$. By \eqref{e:mainbott}, we see that $\delta_z^\varepsilon$ is a periodic orbit of $\tilde X_\varepsilon$ since $\tilde X_\varepsilon$ and $du_\varepsilon\cdot \tilde X_0$ are parallel along $\delta_z^\varepsilon$. Since $\nu_\varepsilon=1+O(\varepsilon)$, we see that the period of $\delta_z^\varepsilon$ is $\tilde T_\varepsilon=T_{\min}+O(\varepsilon)$. Finally, by the claim, the zeros of $V_\varepsilon$ modulo the action of $\Phi_{\tilde X_0}$ are in one-to-one correspondence with the zeros of $dS_\varepsilon$ modulo the action of $\Phi_{\tilde X_0}$. These are in one-to-one correspondence with the zeros of $d\bar S_\varepsilon$. A lower bound for the number of these zeros is $\mathrm{CupLength}(R)+1$, which is equal to $\mathrm{CupLength}(Q)+k_{\tilde a_1}$ by Lemma \ref{l:minsub} and \cite[Equation (7)]{Kerman}.
 \subsection{The Proof of Claim \eqref{e:claim}: The map $\mathcal K_\varepsilon$ is an Isomorphism}\label{sub:iso}
The duality homomorphism $\mathcal K_\varepsilon\colon \ker\tilde\tau|_{\Sigma_{\min}}\to (\ker\tilde\tau|_{\Sigma_{\min}})^*$ is a path of linear maps. Our aim is to prove Claim \eqref{e:claim}, that is, to show that $\mathcal K_\varepsilon$ is invertible for all $\varepsilon>0$ small enough. For this purpose, let us define $\mathcal B$ as the class of paths of homomorphisms 
\begin{equation}
M_\varepsilon\colon \ker\tilde\tau|_{\Sigma_{\min}}\to (\ker\tilde\tau|_{\Sigma_{\min}})^*,\qquad  \forall\,\varepsilon>0 \text{ small enough.}
\end{equation}
Let us define the element
\begin{equation}
N_\varepsilon:=\begin{pmatrix}
        \sigma&0\\
        0&\varepsilon^2\rho
    \end{pmatrix}\in\mathcal B
\end{equation}
and the following two subclasses
\begin{equation}
\begin{aligned}
\mathcal E&:=\left\{M\in\mathcal B\ \Big|\ M_\varepsilon=\begin{pmatrix}
        O(\varepsilon)&O(\varepsilon^2)\\
        O(\varepsilon^2)&O(\varepsilon^3)
    \end{pmatrix}\right\},\\
\mathcal F&:=\left\{M\in\mathcal B\ \Big|\ M_\varepsilon\in N_\varepsilon+\mathcal E\right\},
\end{aligned}
\end{equation}
where the block decomposition
The invertibility of $\mathcal K_\varepsilon$ is based on the following two results.
\begin{lemma}
If $M_\varepsilon\in \mathcal F$, then $M_\varepsilon$ is invertible for all $\varepsilon>0$ small enough.
\end{lemma}
\begin{proof}
We start by observing that $N_\varepsilon$ is invertible for all $\varepsilon>0$. Indeed, $\sigma$ is readily invertible on $\mathcal H$ since $\sigma$ is symplectic on $Q$. Moreover, $\rho$ is invertible on $\ker\tilde\tau|_{\Sigma_{\min,q}}$. This second fact can be seen as follows. By \eqref{e:sigmanondeg}, $T_{\Sigma_{\min,q}}=E_{\min,q}\cap \Sigma_q$ where $E_{\min,q}\subset E_q$ is the $1$-eigenspace of $e^{T_{\min \tilde A_q}}$. Since $e^{T_{\min \tilde A_q}}$ is $\rho$-symplectic, the eigenspace is also $\rho$-symplectic. Since $\tilde X_0$ spans the kernel of $\rho|_{T\Sigma_q}$ and $\ker\tilde\tau$ is transverse to $\tilde X_0$, we see that $\ker\tilde\tau|_{\Sigma_{\min,q}}$ is $\rho$-symplectic, as needed.

After this preliminary observation, we consider
\begin{equation}
\begin{aligned}
M_\varepsilon=N_\varepsilon+\begin{pmatrix}
    O(\varepsilon)& O(\varepsilon^2)\\
    O(\varepsilon^2)&O(\varepsilon^3)
\end{pmatrix}&=N_\varepsilon\left(1+N_\varepsilon^{-1}\begin{pmatrix}
    O(\varepsilon)& O(\varepsilon^2)\\
    O(\varepsilon^2)&O(\varepsilon^3)
\end{pmatrix}\right)\\
&=N_\varepsilon\left(1+\begin{pmatrix}
    O(\varepsilon)& O(\varepsilon^2)\\
    O(1)&O(\varepsilon)
\end{pmatrix}\right)\\
&=N_\varepsilon\begin{pmatrix}
    1+O(\varepsilon)& O(\varepsilon^2)\\
    O(1)&1+O(\varepsilon)
\end{pmatrix}
\end{aligned}
\end{equation}
Therefore, it is enough to show that the right-most matrix in the last step is invertible. Indeed, using the multilinearity of the determinant in the first rows, we get
\begin{equation*}
\det\begin{pmatrix}
1+O(\varepsilon)& O(\varepsilon^2)\\
    O(1)&1+O(\varepsilon)
\end{pmatrix}=\det\begin{pmatrix}
1& 0\\
    O(1)&1+O(\varepsilon)
\end{pmatrix}+O(\varepsilon)=1+O(\varepsilon),
\end{equation*}
which is different from zero if $\varepsilon>0$ is small enough.
\end{proof}
\begin{lemma}
The path $\mathcal K_\varepsilon$ belongs to $T_{\min}\mathcal F$.
\end{lemma}
\begin{proof}
Let $\mathcal W$ be a neighborhood of $\Sigma_{\min}$ inside $\Sigma$. We define analogs of $\mathcal B$, $N$, $\mathcal E$, $\mathcal F$ on $T\Sigma|_{\Sigma_{\min}}$ and on $\mathcal W$, and decorate them with a tilde and a hat, respectively. Thus, let $\tilde{\mathcal B}$ and $\hat{\mathcal B}$ be the classes of paths of homomorphisms 
\begin{equation}
\tilde M_\varepsilon\colon T\Sigma|_{\Sigma_{\min}}\to (T\Sigma|_{\Sigma_{\min}})^*,\qquad  \hat M_\varepsilon\colon T\mathcal W\to T^*\mathcal W,\qquad\varepsilon>0.
\end{equation}
Let 
\begin{equation}
\tilde N_\varepsilon:=\begin{pmatrix}
        \sigma&0\\
        0&\varepsilon^2\rho
    \end{pmatrix}\in\tilde{\mathcal B},\qquad \hat N_\varepsilon:=\begin{pmatrix}
        \sigma&0\\
        0&\varepsilon^2\rho
    \end{pmatrix}\in\hat{\mathcal B}
\end{equation}
and
\begin{equation*}
\begin{aligned}
\tilde{\mathcal E}&:=\left\{\tilde M\in\tilde{\mathcal B}\ \Big|\ \tilde M_\varepsilon=\begin{pmatrix}
        O(\varepsilon)&O(\varepsilon^2)\\
        O(\varepsilon^2)&O(\varepsilon^3)
    \end{pmatrix}\right\},\qquad
\tilde{\mathcal F}:=\left\{\tilde M\in\tilde{\mathcal B}\ \Big|\ \tilde M_\varepsilon\in\tilde N_\varepsilon+\tilde{\mathcal E}\right\};\\
\hat{\mathcal E}&:=\left\{\hat M\in\hat{\mathcal B}\ \Big|\ \hat M_\varepsilon=\begin{pmatrix}
        O(\varepsilon)&O(\varepsilon^2)\\
        O(\varepsilon^2)&O(\varepsilon^3)
    \end{pmatrix}\right\},\qquad
\hat{\mathcal F}:=\left\{\hat M\in\hat{\mathcal B}\ \Big|\ \hat M_\varepsilon\in\hat N_\varepsilon+\hat{\mathcal E}\right\}.
\end{aligned}
\end{equation*}
We will obtain $\mathcal K_\varepsilon$ from $\hat N_\varepsilon$ in subsequent steps and describe the change of class at each step.\medskip

\textit{Step 1: $\omega_\varepsilon\in\hat{\mathcal F}$.} \newline This statement follows from Lemma \ref{l:tildeX}.(c).\medskip

\textit{Step 2: $\mathcal Q_{U_\varepsilon}^*\omega_\varepsilon\in\tilde {\mathcal F}$.}\newline
By Theorem \ref{bottkol_sec_4}, $\mathcal Q_{U_\varepsilon}$ takes values in $T\mathcal W$ for $\varepsilon$ small enough since $\mathcal Q_{U_\varepsilon}=O(1)$. Therefore, we get maps $\mathcal Q_{U_\varepsilon}\colon T\Sigma|_{\Sigma_{\min}}\to T\mathcal W$ and $\mathcal Q_0$ is the inclusion. By \eqref{e:connection}, we have the block decomposition 
\begin{equation}
\mathcal Q_{U_\varepsilon}=\begin{pmatrix}
\mathcal Q_{U_\varepsilon}^\pi&0\\
0&\mathcal Q_{U_\varepsilon}^\nabla\\
\end{pmatrix}=\begin{pmatrix}
O(1)&0\\
0&O(1).
\end{pmatrix}
\end{equation}
Thus, $\mathcal Q_{U_\varepsilon}^*\hat{\mathcal E}\subset\tilde{\mathcal E}$. By Theorem \ref{bottkol_sec_4}, we also know that $U_\varepsilon=O(\varepsilon)$. Hence, a Taylor expansion yields
\begin{equation}
\mathcal Q_{U_\varepsilon}^*\hat N_\varepsilon=\begin{pmatrix}
(\mathcal Q^\pi_{U_\varepsilon})^*\sigma&0\\
0&\varepsilon^2(\mathcal Q^\nabla_{U_\varepsilon})^*\rho
\end{pmatrix}\in \tilde N_\varepsilon+\tilde{\mathcal E}.
\end{equation}
Putting things together, 
\begin{equation}
\mathcal Q_{U_\varepsilon}^*(\hat N_\varepsilon+\hat {\mathcal E})=\mathcal Q_{U_\varepsilon}^*\hat N_\varepsilon+\mathcal Q_{U_\varepsilon}^*\hat {\mathcal E}\in (\tilde N_\varepsilon+\tilde{\mathcal E})+\tilde{\mathcal E}=\tilde N_\varepsilon+\tilde{\mathcal E}.
\end{equation}

\textit{Step 3: $(\mathcal Q_{U_\varepsilon}^*\omega_\varepsilon)\Big(\mathcal Q_{U_\varepsilon}^{-1}\cdot du_\varepsilon,\mathcal Q_{U_\varepsilon}^{-1}\cdot\mathcal P_{U_\varepsilon}\Big)\in\mathcal F$.}\newline
By Theorem \ref{bottkol_sec_4}, we know that $\mathcal Q_{U_\varepsilon}^{-1}\cdot du_\varepsilon$ and $\mathcal Q_{U_\varepsilon}^{-1}\cdot\mathcal P_{U_\varepsilon}$ belong to $\iota+\mathcal D$, where the map $\iota\colon T\Sigma_{\min}\to T\Sigma|_{\Sigma_{\min}}$ is the inclusion and $\mathcal D$ is defined in \eqref{e:mathcald}. By Step 2, it is enough to show that, if $\tilde M_\varepsilon\in\tilde{\mathcal F}$, then \begin{equation}
\tilde M_\varepsilon(\iota+\mathcal D,\iota+\mathcal D)\in\mathcal F.
\end{equation}
Since $\tilde M_\varepsilon(\iota,\iota)\in\mathcal F$, we are left to prove that
\begin{equation}
\tilde M_\varepsilon(\iota,\mathcal D),\, \tilde M_\varepsilon(\mathcal D, \iota),\, \tilde M_\varepsilon(\mathcal D,\mathcal D)\in\mathcal E.
\end{equation}
We check this by matrix multiplication:
\begin{equation}
\tilde M_\varepsilon(\iota,\mathcal D)=\begin{pmatrix}
O(1)&O(\varepsilon^2)\\
O(\varepsilon^2)&O(\varepsilon^2)
\end{pmatrix}\begin{pmatrix}
O(\varepsilon)&O(\varepsilon^2)\\
O(\varepsilon)&O(\varepsilon)
\end{pmatrix}=\begin{pmatrix}
    O(\varepsilon)&O(\varepsilon^2)\\
    O(\varepsilon^2)&O(\varepsilon^3)
\end{pmatrix},
\end{equation}
which is in $\mathcal E$ by definition. Similarly, we get $\tilde M_\varepsilon(\iota,\mathcal D)\in\mathcal E$. Finally,
\begin{equation}
\tilde M_\varepsilon(\mathcal D,\mathcal D)=\begin{pmatrix}
O(\varepsilon)&O(\varepsilon)\\
O(\varepsilon^2)&O(\varepsilon)
\end{pmatrix}\begin{pmatrix}
O(\varepsilon)&O(\varepsilon^2)\\
O(\varepsilon^2)&O(\varepsilon^3)
\end{pmatrix}=\begin{pmatrix}
    O(\varepsilon^2)&O(\varepsilon^3)\\
    O(\varepsilon^3)&O(\varepsilon^4)
\end{pmatrix},
\end{equation}
which is again in $\mathcal E$.\medskip

\textit{Step 4: $(\mathcal Q_{U_\varepsilon}^*\omega_\varepsilon)\Big(\mathcal Q_{U_\varepsilon}^{-1}\cdot du_\varepsilon\cdot d\Phi^t_{\tilde X_0},\mathcal Q_{U_\varepsilon}^{-1}\cdot\mathcal P_{U_\varepsilon}d\Phi^t_{\tilde X_0}\Big)\in\mathcal F, \ \forall\,t\in\mathbb R$.}\newline

By Step 3, it is enough to show that if $M_\varepsilon\in\mathcal F$, then
\begin{equation}
M_\varepsilon(d\Phi_{\tilde X_0}^t,d\Phi_{\tilde X_0}^t)\in\mathcal F.
\end{equation}
Recall that by \eqref{e:dphi} we have
\begin{equation}
d_{(q,v)}\Phi^t_{\tilde X_0}=\begin{pmatrix}
1&0\\
0& e^{t\tilde A_q}
\end{pmatrix}
\end{equation}
and that $e^{t\tilde A_q}$ preserves $\rho_q$ by Lemma \ref{l:connection}. Therefore,
\begin{equation}
M_\varepsilon(d\Phi_{\tilde X_0}^t,d\Phi_{\tilde X_0}^t)=N_\varepsilon(d\Phi_{\tilde X_0}^t,d\Phi_{\tilde X_0}^t)+\mathcal E(d\Phi_{\tilde X_0}^t,d\Phi_{\tilde X_0}^t)=N_\varepsilon+\mathcal E.
\end{equation}

\textit{Step 5: $\displaystyle\mathcal K_\varepsilon=\int_0^{T_{\min}}(\mathcal Q_{U_\varepsilon}^*\omega_\varepsilon)\Big(\mathcal Q_{U_\varepsilon}^{-1}\cdot du_\varepsilon\cdot d\Phi^t_{\tilde X_0},\mathcal Q_{U_\varepsilon}^{-1}\cdot\mathcal P_{U_\varepsilon}d\Phi^t_{\tilde X_0}\Big)dt\in T_{\min}\mathcal F$.}\newline
This statement follows from Step 4 and the fact that if $M_\varepsilon(t)\in \mathcal F$ for $t\in[0,T_{\min}]$, then its average on this interval also belongs to $\mathcal F$.
\end{proof}

\section{Theorem \ref{thmA}, Intermezzo: Proof of Bottkol's Normal Form}\label{sec:proof-bifurcation}
\subsection{The Connection $\tilde\nabla$} The purpose of this section is to prove the Normal Form Theorem \ref{bottkol_sec_4}. Recall the definitions \eqref{e:connection_E} and \eqref{e:connection} of the connections $\tilde\nabla^E$ and $\tilde\nabla$ on $TE$ and $T\Sigma$. We start by describing the torsion $\tilde T$ of the connection $\tilde\nabla$. The result is an adaptation of \cite[Lemma 2]{Dom}.
\begin{lemma}\label{l:torsion}
The kernel of the torsion $\tilde T$ contains the vertical distribution. The range of the torsion $\tilde T$ is contained in the vertical distribution. In other words,
\begin{equation}
\tilde T=\begin{pmatrix}
0&0\\
\tilde T^\nabla_\pi&0
\end{pmatrix}.
\end{equation}
\end{lemma}
\begin{proof}
By \eqref{e:connection_E}, $\tilde\nabla E=\pi^*\nabla^Q\oplus\pi^*\nabla$, where the splitting is with respect to the horizontal-vertical distribution of $\nabla$. The symbol $\nabla$ denotes a connection on $E$ preserving $\gamma$, and $\nabla^Q$ is a connection on $TQ$ such that its torsion $T^Q$ vanishes. By \eqref{e:connection}, $\tilde\nabla$ is the $\gamma$-orthogonal projection $\Pi$ of $\tilde\nabla^E$ from $TE|_{\Sigma}$ to $T\Sigma$. Thus the torsion of the two connections are related by $\tilde T=\Pi\tilde T^{E}$. Therefore, it is enough to prove the statement for $\tilde T^E$. 

Let $\eta,\zeta$ be two vector fields on $Q$. Let $\eta^h,\zeta^h\subset \mathcal H$ be their horizontal lift. Then
\begin{equation}
\begin{aligned}
d\pi\tilde T(\eta^h,\zeta^h)&=d\pi\Big(\pi^*\nabla^Q_{\eta^h}\zeta^h-\pi^*\nabla^Q_{\zeta^h}\eta^h-[\eta^h,\zeta^h]\Big)\\
&=\nabla^Q_{\eta}\zeta-\nabla^Q_{\zeta}\eta-[\eta,\zeta]\\
&=T^Q(\eta,\zeta),
\end{aligned}
\end{equation}
which vanishes. Hence, $\tilde T(\eta^h,\zeta^h)$ is vertical.

The restriction of $\tilde\nabla$ to a fiber $E_q$ is torsion-free as it is equal to $\pi^*\nabla$ and therefore its Christoffel symbols are constant in the vertical direction. Hence, $\tilde T(\mathcal V,\mathcal V)=0$.

Finally, let $\eta$ be a vector field on $Q$ and let $w$ be a section of $E\to Q$. We claim that $\tilde T(\eta^h,w^v)=0$, which will finish the proof. We compute
\begin{equation}
\begin{aligned}
\tilde T(\eta^h,w^v)&=\pi^*\nabla_{\eta^h}w^v-\pi^*\nabla^Q_{w^v}\eta^h-[\eta^h,w^v]\\
&=\pi^*\nabla_{\eta^h}w^v-\pi^*\nabla^Q_{w^v}\eta^h-[\eta^h,w^v]\\
&=(\nabla_{d\pi\eta^h}w)^v-(\nabla_{d\pi w^v}^Q\eta)^h-[\eta^h,w^v].    
\end{aligned}
\end{equation}
Since $d\pi\eta^h=\eta$ and $d\pi w^v=0$, we are left to show
\begin{equation}
(\nabla_\eta w)^v=[\eta^h,w^v].
\end{equation}
Since $\Phi_{w^v}^t(q,v)=(q,v+tw(q))$, we get 
\begin{equation}
d\Phi_{w^v}^t=\begin{pmatrix}
1&0\\
t\nabla w&1
\end{pmatrix}    
\end{equation}
and hence
\begin{equation*}
[\eta^h,w^v]=\frac{d}{dt}\Big|_{t=0}d\Phi_{w^v}^t\eta^h(\Phi^{-t}_{w^v})=\frac{d}{dt}\Big|_{t=0}\begin{pmatrix}
1&0\\
t\nabla w&1    
\end{pmatrix}\begin{pmatrix}
\eta)\\0
\end{pmatrix}=\begin{pmatrix}
0\\\nabla_\eta w
\end{pmatrix}.\qedhere
\end{equation*}
\end{proof}
Let $U$ be any element in $\mathfrak X(\Sigma|_{\Sigma_{\min}})$ and recall the definition of the objects $u:=\exp(U)$, $\delta_U$, $\mathcal P_U$ and $\mathcal Q_U$ from \eqref{e:u}, \eqref{e:delta}, \eqref{e:P}, \eqref{e:Q}. Recall also that $\iota\colon T\Sigma_{\min}\to T\Sigma|_{\Sigma_{\min}}$ is the inclusion. We now give the main estimates about these objects.
\begin{lemma}\label{l:puduxu}
The following expansions around $U=0$ hold
\begin{equation}
\begin{aligned}
(a)&\ \ \mathcal Q_U^{-1}\cdot \mathcal P_U=\iota+\frac{1}{2}\tilde T(U,\cdot)+O(\Vert U\Vert^2),\\
(b)&\ \ \mathcal Q_U^{-1}\cdot du=\iota+\tilde\nabla U+\tilde T(U,\tilde \nabla U)+O(\Vert U\Vert^2),\\
(c)&\ \ \mathcal Q_U^{-1}\cdot X\circ u=X+\tilde\nabla_UX+O(\Vert U\Vert^2),
\end{aligned}
\end{equation}
where $\Vert U\Vert$ denotes the $C^k$-norm of $U$ for any $k\geq 0$ and $X$ is any vector field on $\Sigma$.
\end{lemma}
\begin{proof}
Let $z\in\Sigma_{\min}$ and $\xi\in T_z\Sigma_{\min}$. We recall that 
\begin{equation}
\mathcal P_U(z)\cdot \xi:=d_{(z,U(z))}\exp{} \cdot \xi^{\tilde v}
\end{equation}
This map can be written in terms of Jacobi fields. Indeed, $\mathcal P_U(z)\cdot \xi=J(1)$ where $J$ is the Jacobi vector field of $\tilde \nabla$ along the geodesic $\delta_U$ and with $J(0)=0$ and $\dot J(0)=\xi$. Since $\tilde\nabla$ has torsion, the Jacobi equation is
\begin{equation}
\ddot J=\dot {\tilde T}(\dot\delta,J)+\tilde T(\dot\delta,\dot J)+\tilde R(\dot\delta,J)\dot\delta=0.    
\end{equation}
where $\tilde R$ is the curvature of $\tilde\nabla$. By defining $K:=\mathcal Q^{-1}_UJ$, we obtain a second-order ODE for $K$ given by
\begin{equation}
\begin{cases}
\ddot K= (\tilde \nabla_U(\mathcal Q^{-1}_U \tilde T))(U,K)+(\mathcal Q^{-1}_U\tilde T)(U,\dot K)+(\mathcal Q^{-1}_U\tilde R)(U,K)U,\\
K(0)=0\\
\dot K(0)=\xi.
\end{cases}
\end{equation}
From the equation, we get $\ddot K(0)=\tilde T(U,\xi)$ and therefore
\begin{equation}\label{e:qupu}
\mathcal Q_U^{-1}\mathcal P_U(z)\cdot\xi=K(1)=\xi+\frac{1}{2}\tilde T(U,\xi)+O(\Vert U\Vert^2)\cdot\xi.
\end{equation}
This shows (a). To prove (b), we observe that
\begin{equation}\label{e:duxi}
du\cdot \xi=d_{(z,U(z))}\exp{} \cdot \xi^{\tilde h}+d_{(z,U(z))}\exp{} \cdot (\tilde\nabla_\xi U)^{\tilde v}.
\end{equation}
The summand can be rewritten in terms of Jacobi fields as $d_{(z,U(z))}\exp{} \cdot \xi^{\tilde h}=J(1)$, where $J$ is the Jacobi field of $\tilde\nabla$ along the geodesic $\delta_U$ such that $J(0)=\xi$ and $\dot J(0)=0$.

Defining $K:=\mathcal Q^{-1}J$, we see that $K$ satisfies the ODE above with $K(0)=\xi$, $\dot K(0)=0$. Using the ODE, we obtain $\ddot K(0)=O(\Vert U\Vert^2)$ and therefore
\begin{equation}
\mathcal Q_U^{-1}d_{(z,U(z))}\exp{} \cdot \xi^{\tilde h}=\xi+O(\Vert U\Vert^2).
\end{equation}
The second summand of \eqref{e:duxi} can be expanded using \eqref{e:qupu} with $\tilde\nabla_\xi U$ instead of $\xi$. We arrive at the expansion (b):
\begin{equation}
\mathcal Q^{-1}_U\cdot du\cdot \xi=\xi+O(\Vert U\Vert^2)+\tilde \nabla_\xi U+\tilde T(U,\tilde\nabla_\xi U)+O(\Vert U\Vert^2).
\end{equation}

Finally, the expansion (c) stems from the fact that $\mathcal Q_U$ is the parallel transport with respect to $\tilde \nabla$.
    \end{proof}
\subsection{The Normal Form via the Inverse Function Theorem}
Let us consider the setting of Theorem \ref{bottkol_sec_4}. Let $\tilde X_\varepsilon=\tilde X_0+Y_\varepsilon$ with $Y_\varepsilon$ as in \eqref{e:tildeY}. We aim at showing the existence of $U_\varepsilon$, $V_\varepsilon$ and $\nu_\varepsilon=1+\mu_\varepsilon$ satisfying the equation
\begin{equation}\label{e:nf}
\nu_\varepsilon \tilde X_\varepsilon\circ u_\varepsilon +\mathcal P_{U_\varepsilon}\cdot V_\varepsilon-du_\varepsilon\cdot \tilde X_0=0.
\end{equation}
Since $\tilde Y_\varepsilon=O(\varepsilon)$, the existence of such objects with $\mu_\varepsilon=O(\varepsilon)$, $U_\varepsilon=O(\varepsilon)$ and $V_\varepsilon=O(\varepsilon)$ follows from an application of the Inverse Function Theorem due to Bottkol \cite{Bottkol}, see also \cite{Kerman}, \cite[Appendix B]{AB_21} or \cite{sanjay2025c2localsystolicoptimalityzoll}. The only small difference is that in the references the connection $\tilde \nabla$ is chosen to be torsion-free, while here we made a different choice to ensure that the connection preserves the horizontal-vertical splitting. However, this does not affect the result, as we now show with a computation that will also be needed to prove the remaining estimates of Theorem \ref{bottkol_sec_4}.

The Inverse Function Theorem takes place between the Banach spaces $\mathbb A$ and $\mathbb B$. The space $\mathbb A$ is the completion of
\begin{equation*}
\bar C^\infty(\Sigma_{\min})\times \mathfrak X_1(\Sigma|_{\Sigma_{\min}})\times \bar{\mathfrak X}(\ker\tilde\tau|_{\Sigma_{\min}})
\end{equation*}
with respect to the norm
\begin{equation}
\Vert (\mu,U,V)\Vert:=\Vert\mu\Vert+\Vert U\Vert+\Vert[\tilde X_0,U]\Vert+\Vert V\Vert,
\end{equation}
where on the right we take $C^k$-norms for a fixed $k\geq 1$. Below, we will denote by the symbol $\mathcal R$ terms that are quadratic in this norm. The space $\mathbb B$ is the completion of $\mathfrak X(\Sigma|_{\Sigma_{\min}})$ with respect to the $C^k$-norm.

We start by applying $\mathcal Q_{U_\varepsilon}^{-1}$ on both sides of 
\eqref{e:nf}. By Lemma \ref{l:puduxu}, we get
\begin{equation*}
\nu_\varepsilon \tilde X_\varepsilon+\nu_\varepsilon\tilde \nabla_{U_\varepsilon}\tilde X_\varepsilon-\tilde X_0-\tilde\nabla_{\tilde X_0}U_\varepsilon-\tilde T(U_\varepsilon,\tilde \nabla_{\tilde X_0} U_\varepsilon)+V_\varepsilon+\frac12\tilde T(U_\varepsilon,V_\varepsilon) =\mathcal R_\varepsilon.
\end{equation*}
We now analyze the terms on the left hand-side. The term $\tilde T(U_\varepsilon,V_\varepsilon)$ is quadratic in $\Vert(\mu_\varepsilon, U_\varepsilon,V_\varepsilon)\Vert$. Moreover, 
\begin{equation}
\tilde T(U_\varepsilon,\tilde \nabla_{\tilde X_0} U_\varepsilon)=\tilde T(U_\varepsilon,[\tilde X_0,U_\varepsilon]),
\end{equation}
where we have used that $\tilde T(\tilde X_0,U_\varepsilon)=0$ and that $\nabla_{U_\varepsilon}\tilde X_0$ is vertical and hence annihilates $\tilde T(U_\varepsilon,\cdot)$, see Lemma \ref{l:torsion}. Thus, $\tilde T(U_\varepsilon,V_\varepsilon)$ is also quadratic in $\Vert( \mu_\varepsilon,U_\varepsilon,V_\varepsilon)\Vert$. We expand
\begin{equation}
\nu_\varepsilon\tilde X_\varepsilon-\tilde X_0=\nu_\varepsilon\tilde X_0+\nu_\varepsilon Y_\varepsilon-\tilde X_0=\mu_\varepsilon\tilde X_0+\nu_\varepsilon Y_\varepsilon
\end{equation}
and
\begin{equation}
\begin{aligned}
\nu_\varepsilon\tilde\nabla_{U_\varepsilon}\tilde X_\varepsilon-\tilde \nabla_{\tilde X_0}U_\varepsilon&=\tilde\nabla_{U_\varepsilon}\tilde X_0+\tilde\nabla_{U_\varepsilon}Y_\varepsilon+\mu_\varepsilon \tilde\nabla_{U_\varepsilon}\tilde X_\varepsilon-\tilde\nabla_{\tilde X_0}U_\varepsilon\\
&=-[\tilde X_0,U_\varepsilon]+\tilde\nabla_{U_\varepsilon}Y_\varepsilon+\mu_\varepsilon \tilde\nabla_{U_\varepsilon}\tilde X_\varepsilon,
\end{aligned}
\end{equation}
where we have used that $\tilde T(U_\varepsilon,\tilde X_0)=0$ as $\tilde X_0$ is vertical. Since $\mu_\varepsilon \tilde\nabla_{U_\varepsilon}\tilde X_\varepsilon$ is quadratic in $\Vert( \mu_\varepsilon,U_\varepsilon,V_\varepsilon)\Vert$, we see that \eqref{e:nf} can be rewritten as
\begin{equation}
\mu_\varepsilon\tilde X_0-[\tilde X_0,U_\varepsilon]+V_\varepsilon=-\nu_\varepsilon Y_\varepsilon-\tilde\nabla_{U_\varepsilon}Y_\varepsilon+\mathcal R_\varepsilon.
\end{equation}
We define
\begin{equation}
\Lambda\colon \mathbb A\to\mathbb B,\qquad \Lambda(\mu,U,V):=\mu\tilde X_0-[\tilde X_0,U]+V.
\end{equation}
By Lemma \ref{l:vectors}, this map is an isomorphism with inverse $\Lambda^{-1}(Z)$ given by
\begin{equation}\label{e:inverse}
\mu=\tilde\tau(\overline Z),\qquad U=\mathcal L^{-1}(Z-\overline{Z}),\qquad V=\overline{Z}-\tilde\tau(\overline Z)\tilde X_0.
\end{equation}
Since the terms $-\nu_\varepsilon Y_\varepsilon-\tilde\nabla_{U_\varepsilon}Y_\varepsilon+\mathcal R_\varepsilon$ are continuous with respect to the parameter $Y_\varepsilon$ in the $C^{k+1}$-topology and the linear terms $-\nu_\varepsilon Y_\varepsilon-\tilde\nabla_{U_\varepsilon}Y_\varepsilon$ vanish for $Y_\varepsilon=0$, an application of the Inverse Function Theorem with parameter implies that there are unique solutions $\mu_\varepsilon$, $U_\varepsilon$, $V_\varepsilon$ with 
\begin{equation}\label{e:norm}
\Vert(\mu_\varepsilon,U_\varepsilon,V_\varepsilon)\Vert=O(\Vert Y_\varepsilon\Vert_{C^{k+1}})=O(\varepsilon),
\end{equation}
where the last equality follows from \eqref{e:tildeY}.

To finish the proof of Theorem \ref{bottkol_sec_4}, we have to establish the three estimates in \eqref{e:estimates}.(b). By \eqref{e:inverse}, we have
\begin{equation}\label{e:uepsilon}
U_\varepsilon=\mathcal L^{-1}(Z_\varepsilon-\overline{Z_\varepsilon}),\qquad Z_\varepsilon:=-\nu_\varepsilon Y_\varepsilon-\tilde\nabla_{U_\varepsilon}Y_\varepsilon+\mathcal R_\varepsilon.
\end{equation}
By \eqref{e:norm},
\begin{equation}
-\tilde\nabla U_\varepsilon Y_\varepsilon+ \mathcal R_\varepsilon=O(\varepsilon^2),
\end{equation}
since $\mathcal R_\varepsilon$ is quadratic in the norm of $\mathbb A$. Moreover, by \eqref{e:tildeY} we have
\begin{equation}
-\nu_\varepsilon Y^\pi_\varepsilon=O(\varepsilon^2).
\end{equation}
Since $\mathcal L^{-1}$ preserves the horizontal-vertical splitting by Lemma \ref{l:vectors}, we conclude by \eqref{e:uepsilon} that
\begin{equation}\label{e:oepsilon2}
U_\varepsilon^\pi=O(\varepsilon^2).
\end{equation}
This shows the first estimate.

For the second estimate, we recall by Lemma \ref{l:puduxu} that
\begin{equation*}
\begin{aligned}
\mathcal Q_{U_\varepsilon}^{-1}\cdot du_{\varepsilon}-\iota=\tilde\nabla U_\varepsilon+\tilde T(U_\varepsilon,\tilde \nabla U)+O(\Vert U_\varepsilon\Vert^2)&=\tilde\nabla U_\varepsilon+O(\varepsilon^2)\\
&=\begin{pmatrix}
O(\varepsilon^2)& O(\varepsilon^2)\\
O(\varepsilon)& O(\varepsilon)
\end{pmatrix}\in\mathcal D,
\end{aligned}
\end{equation*}
where we used \eqref{e:norm}, \eqref{e:oepsilon2} and the definition \eqref{e:mathcald} of $\mathcal D$. This shows the second estimate.

For the third estimate, we recall by Lemma \eqref{l:puduxu} that
\begin{equation*}
\mathcal Q_{U_\varepsilon}^{-1}\cdot \mathcal P_{U_\varepsilon}-\iota=\frac{1}{2}\tilde T(U_\varepsilon,\cdot)+O(\Vert U_\varepsilon\Vert^2)=\begin{pmatrix}
0&0\\
\tfrac12 \tilde T^\nabla_\pi&0
\end{pmatrix}\begin{pmatrix}
O(\varepsilon^2)\\ O(\varepsilon)
\end{pmatrix}+O(\varepsilon^2)\in\mathcal D,
\end{equation*}
where we used \eqref{e:norm}, \eqref{e:oepsilon2} and the definition \eqref{e:mathcald} of $\mathcal D$. This shows the third estimate, and finishes the proof of Theorem \ref{bottkol_sec_4}.\hfill\qed

\section{Theorem \ref{thmA}, Step 3: the Conformally Compatible Case}\label{sec:conformally-compatible}
In this section we conclude the proof of our first main theorem. 
\begin{proof}[Proof of Theorem \ref{thmA}]
Assume that $H$ is Zoll along a sequence of energies $\tfrac12\varepsilon^2_n$ converging to the minimum. By Corollary \ref{cor: bifurcation}, we know that
\begin{equation}
\tilde A=J=\frac{A}{a},
\end{equation}
where $J$ is an almost complex structure which is $\rho$-compatible and we have defined the function $a:=(\det A)^{-\frac{1}{2k}}\colon Q\to(0,\infty)$. Thus
\begin{equation}
\gamma_q(w,v)=\rho_q(w,A_qv)=a(q)\rho_q(w,J_q v),\qquad\forall\,q\in Q,\ \forall\,w,v\in E_q.
\end{equation}
Let us choose a connection $\nabla$ on $\pi\colon E\to Q$ such that $\nabla\rho=0$ and $\nabla J=0$, which exists since $J$ is $\rho$-compatible. 
By Lemma \ref{l:tildeX}
\begin{equation}\label{eq: expansion Y}
\tilde X_\varepsilon=a X_\varepsilon=\tilde X_0+\tilde Y_\varepsilon,\qquad \tilde Y_\varepsilon^\pi=\varepsilon^2a(q)\zeta_q(v,v)+O(\varepsilon^3).
\end{equation}
According to \eqref{eq:zeta-formula}, for every $(q,v)\in \Sigma$ and $\eta\in T_qQ$, we get
\begin{equation}
\begin{aligned}
\sigma_q(a(q)\zeta_q(v,v),\eta)=\tfrac12a(q) (\nabla_\eta \gamma)_q(v,v)&=\tfrac12 a(q)(d_q a\cdot \eta) \rho_q(v,J_qv)\\
&=\tfrac{1}{2}(d_q a\cdot \eta)\gamma_q(v,v)\\
&=\frac12d_q
a\cdot \eta,
\end{aligned}
\end{equation}
where we used that $\nabla\rho=0$, $\nabla J=0$, and $\gamma_q(v,v)=1$.
Thus, we conclude that 
\begin{equation}\label{eq: Hamvecfora}
a(q)\zeta_q(v,v)=X_{-a}(q),\qquad\forall\,(q,v)\in \Sigma,
\end{equation}
where $X_{-a}$ is the Hamiltonian vector field on $Q$ of the Hamiltonian function $-a$ with respect to the symplectic form $\sigma$. We now claim that $a$, is a constant function. Assume $a$ it is not constant and let $q_0\in Q$ be such that $d_{q_0}a\neq0$. Denote with $q\colon \mathbb R\to Q$ the orbit of $X_{-a}$ with $q(0)=q_0$. Since $X_{-a}(q_0)\neq 0$, there exists $t_0>0$ such that the $q|_{[0,t_0]}$ is an embedding. Let now $z_\varepsilon\colon\mathbb R\to\Sigma$ be any flow line of $\tilde X_\varepsilon$ such that $\pi(z_\varepsilon(0))=q_0$ and denote $q_\varepsilon:=\pi(z_\varepsilon)\colon \mathbb R\to Q$. By \eqref{eq: expansion Y} and \eqref{eq: Hamvecfora}, we have that
\begin{equation}
\dot q_\varepsilon(t)=\varepsilon^2 X_{-a}(q_\varepsilon(t))+O(\varepsilon^3).
\end{equation}
Therefore, from the continuous dependence of the solutions to ordinary differential equations from the vector field, we conclude that
\begin{equation}
d(q_\varepsilon(s),q(\varepsilon^2s))=O(\varepsilon),\qquad \forall\,s\in[0,t_0/\varepsilon^2].
\end{equation}
Since $q|_{[0,t_0]}$ is an embedding, we conclude that the period of $q_\varepsilon$, and hence, of $z_\varepsilon$ is at least $t_0/\varepsilon^2$. However, we know that when $\varepsilon=\varepsilon_n$, then, as $n$ tends to infinity, the period $T_n$ of $z_{\varepsilon_n}$ converges to $2\pi$ by Corollary \ref{cor: bifurcation}. Since $t_0/\varepsilon^2_n$ diverges to infinity, we reach a contradiction. We have thus shown that the function $a=\det A^{-\frac{1}{2k}}$ is constant. By the normalization \eqref{normalization_of_A_morse-bott}, we deduce that $a=1$. Thus $A=J$ is an almost complex structure (compatible with $\rho$) and the proof is complete.
\end{proof}

\section{Theorem \ref{thmB}, Step 1: The Magnetic Normal Form}\label{sec:proof-thmB}

Throughout this section let $(Q,g,\beta)$ be a magnetic system with $\beta$ symplectic, and let $\omega_{\mathrm{can},\beta}=d\lambda-\pi^*\beta$ be the twisted symplectic form on $TQ$ (via the metric identification $T^*Q\simeq TQ$). We write $H^g(q,v)=\tfrac12\,g_q(v,v)$. As recalled in the Introduction, the linearized dynamics in the symplectic normal bundle is generated by the Lorentz endomorphism $B$ determined by $g(Bu,w)=\beta(u,w)$. Assume that $H^g$ is Zoll along a sequence $\varepsilon_n\to 0$, i.e.\ each $\Sigma_{\varepsilon_n}$ is Zoll up to a global smooth time reparametrization.  Applying Theorem \ref{thmA} to this symplectic Morse--Bott minimum yields that the fiberwise Hessian is $\rho$-compatible. Hence $B=J$ is a $\beta$-compatible almost complex structure, that is, $(g,\beta)$ is an almost Kähler structure. To finish the proof, we need a normal form for the magnetic flow associated to an almost Kähler structure. This is the content of the next subsection.

\subsection{A normal form for almost K\"ahler magnetic flows}

Assume from now on that $(g,\beta)$ is an almost K\"ahler structure. Equivalently, the Lorentz endomorphism is given by an almost complex structure $J$, and $\beta = g(J\cdot,\cdot)$. For every $\varepsilon>0$, let $\Sigma_\varepsilon=(H^g)^{-1}\!\left(\tfrac12\varepsilon^2\right)$ be the corresponding energy level. Via the rescaling map $SQ\longrightarrow \Sigma_\varepsilon,\ (q,v)\longmapsto (q,\varepsilon v)$, the restriction of the twisted symplectic form identifies with
\[
\omega_\varepsilon:=\varepsilon\,d\lambda-\pi^*\beta
\]
on $SQ$. We now fix the Chern connection $\nabla=\nabla^{g,J}$, namely the unique connection characterized by
\[
\nabla g=0,\qquad \nabla J=0,\qquad T(J\cdot,\cdot)=T(\cdot,J\cdot),
\]
where $T$ denotes the torsion tensor of $\nabla$. Since we are in the almost K\"ahler setting, one has $T=-\tfrac14 N^J$, with $N^J$ the Nijenhuis tensor. As a consequence, $T$ is complex antilinear in both entries. The main result of this subsection is the following normal form theorem.
\begin{theorem}\label{thm: magnetic normal form}
There exists an isotopy $\psi_\varepsilon: SQ\to SQ$ with $\psi_0=\mathrm{id}$ and 
\begin{equation}\label{isotopy}
\psi^*_\varepsilon\omega_\varepsilon=-\pi^*\beta+d \left( H_\varepsilon \tau\right)+o(\varepsilon^4),
\end{equation}
where $\tau=\tau^{\beta,\nabla^{g,J}}$ is the angular form defined in \ref{def_of_coupling_1_form}, $\mu\colon SQ\to P_\C(TQ)$ is the circle-bundle projection, and the function $H_\varepsilon$ is given by
\begin{equation}\label{eq:H-eps-expansion}
H_\varepsilon=\frac{\varepsilon^2}{2}+\frac{\varepsilon^4}{4}\,\frac{\hat K\circ\mu}{2},
\qquad
\hat K:=K-\tfrac{2}{3}|T^*|^2=K-\tfrac{1}{24}|N^*|^2,
\end{equation}
where $K=K^{g,J}$ is the holomorphic sectional curvature of the Chern connection, and $|N^*|^2:P_{\mathbb C}(TQ)\to \lbrack 0,\infty)$ is the fiberwise function defined by
\begin{equation}\label{eq:def-Nstar}
|N^*|^2(\mu_q(v))
:=\big|N_v^*v\big|_g^2
\qquad (v\in S_qQ),
\end{equation}
with $N_v:T_qQ\to T_qQ$ given by $N_v(w):=N^J(v,w)$ and $N_v^*$ its $g$-adjoint. 
\end{theorem}
Before carrying out the proof in the next subsection, we fix some notation and prove a couple of useful lemmas. We denote by $\sH\oplus\sV$ the horizontal--vertical splitting of $TTQ$ induced by $\nabla$, and we use the notation $(\cdot)^{h}$, $(\cdot)^{v}$ for horizontal/vertical lifts, and $(\cdot)^{\pi}$, $(\cdot)^{\nabla}$ for horizontal/vertical projections. The canonical symplectic form and the differential of the coupling form are given by
\begin{equation}
d\lambda_{(q,v)}=\begin{pmatrix}
g_q(v,T_q(\cdot,\cdot))&-g_q\\ g_q&0.
\end{pmatrix},\qquad d\tau_{(q,v)}=\begin{pmatrix}
\beta_q(R^\nabla_q(\cdot,\cdot)v,v)&0\\ 0& 2\beta_q
\end{pmatrix}
\end{equation}
in the horizontal-vertical splitting. We now establish two formulas that will be used in the proof of the normal form theorem.
\begin{lemma}\label{lem: expansion differential form}
Assume that $\psi_\varepsilon\colon \Sigma\to\Sigma$ is a path of diffeomorphisms starting at the identity generated by a non-autonomous vector field $Z_\varepsilon$, that is,
\begin{equation}
\dot\psi_\varepsilon=Z_\varepsilon\circ\psi_\varepsilon,\qquad\psi_0=\mathrm{id}.
\end{equation}
For every differential form $\alpha$ on $\Sigma$ we have
\begin{equation}
(\psi_\varepsilon^{-1})^*\alpha=\alpha+\varepsilon \alpha'+\frac{\varepsilon^2}{2}\alpha''+o(\varepsilon^2),
\end{equation}
where 
\begin{equation}
\alpha'=-\mathcal L_{Z_0}\alpha,\qquad \alpha''=\mathcal L_{Z_0}\mathcal L_{Z_0}\alpha-\mathcal L_{\dot Z_0}\alpha  
\end{equation} 
\end{lemma}
\begin{proof}
The path $\psi^{-1}_\varepsilon$ is generated by the vector field $Y_\varepsilon$ given by the formula
\begin{equation}
Y_\varepsilon=-d\psi_\varepsilon^{-1}(Z_\varepsilon\circ\psi_\varepsilon).
\end{equation}
In particular, $Y_0=-Z_0$ and $\dot Y_0=-\dot Z_0-[Z_0,Z_0]=-\dot Z_0$. Thus, we compute
\begin{equation}
\alpha'=\frac{d}{d\varepsilon}\Big|_{\varepsilon=0}(\psi^{-1}_\varepsilon)^*\alpha=(\psi_0^{-1})^*\mathcal L_{Y_0}\alpha=-\mathcal L_{Z_0}\alpha
\end{equation}
and
\begin{equation}
\begin{aligned}
\alpha''=\frac{d^2}{d\varepsilon^2}\Big|_{\varepsilon=0}(\psi^{-1}_\varepsilon)^*\alpha&=\frac{d}{d\varepsilon}\Big|_{\varepsilon=0}(\psi_\varepsilon^{-1})^*\mathcal L_{Y_\varepsilon}\alpha\\
&=(\psi_0^{-1})^*\mathcal L_{Y_0}\mathcal L_{Y_0}\alpha+(\psi_0^{-1})^*\mathcal L_{\dot Y_0}\alpha\\
&=\mathcal L_{Z_0}\mathcal L_{Z_0}\alpha-\mathcal L_{\dot Z_0}\alpha.
\end{aligned}
\end{equation}
\end{proof}
\begin{lemma}\label{lem: Z0 identity}
If we define
\[
Z_0=(-Jv)^h+(-\tfrac13JT^*_vv)^v,
\]
then
\begin{equation}
(d\iota_{Z_0}d\tau)(Z_0,(Jv)^v)=\hat K(\mu(v)).
\end{equation}
\end{lemma}
\begin{proof}
We have
\begin{equation}
\iota_{(Jv)^v}\iota_{Z_0}d\iota_{Z_0}d\tau=-\iota_{Z_0}\mathcal L_{(Jv)^v}\iota_{Z_0}d\tau, 
\end{equation}
where we used that $\iota_{(Jv)^v}\iota_{Z_0}d\tau=0$. By the definition of $Z_0$ we get
\[
\iota_{Z_0}d\tau=\beta(R(\ \cdot^\pi,Jv)v,v)-2\beta(\tfrac13 JT^*_vv,\cdot^\nabla).
\]
We differentiate the two pieces with respect to $(Jv)^v$ separately. For the first piece we use that the the flow of $(Jv)^v$ act as the identity of the horizontal distribution. Thus,
\begin{equation}
\begin{aligned}
\mathcal L_{(Jv)^v}\beta(R(\ \cdot^\pi,Jv)v,v)&=\frac{d}{dt}\Big|_{t=0}\beta(R(\ \cdot^\pi,Je^{tJ}v)e^{tJ}v,e^{tJ}v)\\
&=-\beta(R(\ \cdot^\pi,v)v,v)\\
&\quad\, +\beta(R(\ \cdot^\pi,Jv)Jv,v)+\beta(R(\ \cdot^\pi,Jv)v,Jv).
\end{aligned}
\end{equation}
Applying $-\iota_{Z_0}$ and using that $(Z_0)^\pi=-Jv$ we get
\begin{equation}
-\iota_{Z_0}\mathcal L_{(Jv)^v}\beta(R(\ \cdot^\pi,Jv)v,v)=-\beta(R(Jv,v)v,v)=g(R(Jv,v)v,Jv)=K(\mu(v)).
\end{equation}
We now differentiate the second piece with respect to $(Jv)^v$. Using that the flow of $(Jv)^v$ act as multiplication by $e^{tJ}$ on the vertical distribution and that the torsion tensor is anticomplex, we find
\begin{equation}
\begin{aligned}
\mathcal L_{(Jv)^v}\beta(-\tfrac23 JT^*_vv,\cdot^\nabla)&=-\tfrac23\frac{d}{dt}\Big|_{t=0}\beta(JT^*_{e^{tJ}v}e^{tJ}v,e^{tJ}\cdot^\nabla)\\
&=-\tfrac23\frac{d}{dt}\Big|_{t=0}\beta(e^{-tJ}Je^{-tJ}e^{-tJ}T^*_{v}v,\cdot^\nabla)\\
&=-\tfrac23\frac{d}{dt}\Big|_{t=0}\beta(e^{-3tJ}JT^*_{v}v,\cdot^\nabla)\\
&=-2\beta(T^*_{v}v,\cdot^\nabla).
\end{aligned}
\end{equation}
Applying $-\iota_{Z_0}$ and using that $Z_0^\nabla=-\tfrac13JT^*_vv$, we get
\begin{equation}
-\iota_{Z_0}\mathcal L_{(Jv)^v}\beta(-\tfrac23 JT^*_vv,\cdot^\nabla)=2\beta(T^*_vv,-\tfrac13JT^*_vv)=-\tfrac23|T^*_vv|^2=-\tfrac{2}{3}|T^*|^2(\mu(v)).
\end{equation}
\end{proof}

\subsection{The proof of the Normal Form Theorem \ref{thm: magnetic normal form}}\label{subsec:moser-expansion}
Let $Z_\varepsilon$ be the non-autonomous vector field generating $\psi_\varepsilon$,
\[
\dot\psi_\varepsilon = Z_\varepsilon\circ\psi_\varepsilon,\qquad \psi_0=\mathrm{id}.
\]
Differentiating \eqref{isotopy} with respect to $\varepsilon$ we obtain:
\begin{equation}
    \psi^*_{\varepsilon}\big(d\lambda + \mathcal{L}_{Z_{\varepsilon}}\omega_{\varepsilon} \big) = d(h_{\varepsilon}\tau) +o(\varepsilon^3), \qquad h_{\varepsilon}:= \dot H_{\varepsilon}\circ\mu.
\end{equation}
By Cartan's formula, the above equation then reduces to:
\[
\psi_\varepsilon^* \left[ d{\iota_{Z_\varepsilon}} \omega_\varepsilon + d \lambda \right] = d(h_\varepsilon \tau)+o(\varepsilon^3)
\tag{$\star$}
\]
Since the exterior differential commutes with pull-back, the above equation is satisfied if we can solve
\begin{equation}\label{eq:basic-heart}
\iota_{Z_\varepsilon}\omega_\varepsilon + \lambda
=
(\psi^{-1}_\varepsilon)^*(h_\varepsilon\,\tau)
+o(\varepsilon^3).\tag{$\heartsuit$}
\end{equation}
We make the Ansatz
\begin{equation}
\begin{aligned}
Z_\varepsilon&=Z_0+\varepsilon Z_0'+\varepsilon^2 Z_0''+\varepsilon^3 Z_0''',\qquad Z_\varepsilon^\nabla\in\{v,Jv\}^\perp,\\
h_\varepsilon&=h_0+\varepsilon h_0'+\varepsilon^2 h_0''+\varepsilon^3 h_0'''.
\end{aligned}
\end{equation}
We construct the solutions in increasing order of $\varepsilon$. 
\medskip

\textit{Zeroth order.} Let $Z_\varepsilon := Z_0 + \varepsilon Z'_\varepsilon$, $h_\varepsilon = h_0 + \varepsilon h'_\varepsilon$. Then evaluating $(\heartsuit)$ at $\varepsilon=0$ implies
\[
-\iota_{Z_0} \pi^* \beta + \lambda = h_0 \tau.
\tag{$\infty$}
\]
The vertical part gives
\[
0=h_0\beta(v,\cdot)\qquad\Longleftrightarrow\qquad h_0=0.
\]
The horizontal part gives
\[
-\beta(Z_0^\pi,\cdot)+g(v,\cdot)=0 \qquad
\Longleftrightarrow \qquad-g(J Z_0^\pi,\cdot)+g(v,\cdot)=0,
\]
which has the solution $ Z_0^\pi=-Jv$. Thus, for every $Z_0^\nabla\in \{v,Jv\}^\perp$, we get the solution
\[
h_0=0, \qquad Z_0=(-Jv)^h+(Z_0^\nabla)^v.
\]
Substituting back into $(\heartsuit)$ and dividing by $\varepsilon$
\[
\iota_{Z_0} d\lambda + \iota_{Z'_\varepsilon} \omega_\varepsilon = (\psi_\varepsilon^{-1})^* (h'_\varepsilon \tau) \tag{$\heartsuit$'}+o(\varepsilon^2).\vspace{10pt}
\]

\textit{First order:} Let $Z'_\varepsilon := Z'_0 + \varepsilon Z''_\varepsilon$, 
$h'_\varepsilon = h'_0 + \varepsilon h''_\varepsilon$. Evaluating ($\heartsuit$') at $\varepsilon=0$ implies
\[
\iota_{Z_0} d\lambda - \iota_{Z'_0} \pi^* \beta= h'_0 \tau \tag{$\infty'$}
\]
The vertical part gives
\[
-g( Z_0^\pi,\cdot)=h'_0\beta(v,\cdot)\qquad\Longleftrightarrow\qquad g(Jv,\cdot)=h'_0g(Jv,\cdot),
\]
which has the solution $h_0'=1$. The horizontal part gives
\[
-g(v,T_{Jv}\cdot)+g(Z_0^\nabla,\cdot)-\beta( (Z_0')^\pi,\cdot)=0,
\]
which has the solution $ (Z_0')^\pi=T^*_vv-JZ_0^\nabla$.
Thus, for every $(Z_0')^\nabla\in\{v,Jv\}^\perp$ we get the solution
\[
h'_0=1,\qquad Z_0'=(T^*_vv-JZ_0^\nabla)^h+((Z'_0)^{\nabla})^v.
\]
Substituting back into \((\heartsuit)'\) and dividing by $\varepsilon$:
\[
\iota_{Z'_0} d\lambda + \iota_{Z''_\varepsilon} \omega_\varepsilon 
= \tau_\varepsilon' + (\psi_\varepsilon^{-1})^*(h''_\varepsilon \tau)+o(\varepsilon),
\tag{$\heartsuit$''}
\]
where $\tau_\varepsilon'$ is uniquely defined by the equation $(\psi_\varepsilon^{-1})^*\tau=\tau+\varepsilon\tau_\varepsilon'$ and by Lemma \ref{lem: expansion differential form} we have 
\begin{equation}\label{eq1}
    \tau_0'=-\mathcal L_{Z_0}\tau=- d(g(Jv,Z_0^\nabla)) -\iota_{Z_0}d\tau=-\iota_{Z_0}d\tau.\vspace{10pt}
\end{equation}

\textit{Second order:} Let $Z''_\varepsilon := Z''_0 + \varepsilon Z'''_\varepsilon$, 
$h''_\varepsilon = h''_0 + \varepsilon h'''_\varepsilon$. Evaluating ($\heartsuit$'') at $\varepsilon=0$ implies:
\[
\iota_{Z'_0} d\lambda- \iota_{Z''_0} \pi^* \beta  
= -\iota_{Z_0}d\tau+ h''_0 \tau.
\tag{$\infty''$}
\]
The vertical part gives
\[
-g(T^*_vv-JZ_0^\nabla,\cdot)=-2\beta(Z_0^\nabla,\cdot)+h''_0\beta(v,\cdot),
\]
which can be rewritten as
\[
\beta(JT^*_vv+3Z_0^\nabla,\cdot)=h_0''\beta(v,\cdot).
\]
Choosing 
\[
Z_0^\nabla=-\tfrac13 JT^*_vv,
\]
we get $h''_0=0$. Note that this is a valid choice since 
\[
g(JT^*_vv,v)=0,\qquad g(JT^*_vv,Jv)=0.
\]
The horizontal part gives:
\[
g(v,T(T^*_vv-JZ_0^\nabla,\cdot))+g((Z'_0)^\nabla,\cdot)-\beta( (Z_0'')^\pi,\cdot)=\beta(R(Jv,\cdot)v,v),
\]
which can be rewritten as
\[
g(\tfrac23T^*_{T^*_vv},\cdot)+g((Z_0')^\nabla,\cdot)-g(J (Z_0'')^\pi,\cdot)=-g(R^*_{Jv,v}Jv,\cdot),
\]
where we defined $R_{Jv,v}w=R(Jv,w)v$. It follows that
\[
(Z_0'')^\pi=-\tfrac23 JT^*_{T^*_vv}v-J(Z'_0)^\nabla-JR^*_{Jv,v}Jv.
\]
For every $(Z_0'')^\nabla\in\{v,Jv\}^\perp$, we have the solution
\[
h''_0=0,\qquad Z_0''=(-\tfrac23 JT^*_{T^*_vv}v-J(Z'_0)^\nabla-JR^*_{Jv,v}Jv)^h+((Z_0'')^\nabla)^v.
\]
Substituting back into \((\heartsuit'')\) and dividing by $\varepsilon$ yields:
\[
 \iota_{Z''_0} d\lambda + \iota_{Z'''_\varepsilon} \omega_\varepsilon
=\tfrac12\tau_\varepsilon'' + \psi_\varepsilon^*(h'''_\varepsilon \tau) + o(1)
\tag{$\heartsuit'''$} 
\]
where $\tau_\varepsilon''=\tau''_0+\varepsilon\tau_\varepsilon'''$. By Lemma \ref{lem: expansion differential form} $\tau''_0=\sL_{Z_0}\sL_{Z_0}\tau-\sL_{Z_0'}\tau$. Since $Z_0,Z_0'\in\ker\tau$ the formula reduces to
\[
\tau''_0=\iota_{Z_0}d\iota_{Z_0}d\tau-\iota_{Z_0'}d\tau.\vspace{10pt}
\]

\textit{Third order:} Let $Z'''_\varepsilon := Z'''_0$, 
$h'''_\varepsilon = h'''_0$. Evaluating ($\heartsuit'''$) at $\varepsilon=0$ implies:
\[
\iota_{Z''_0} d\lambda - \iota_{Z'''_0} \pi^* \beta 
=\tfrac12\tau''_0+ h'''_0 \tau
\tag{$\infty'''$}
\]
The vertical part of \((\infty''')\) yields:
\[
-g((Z_0'')^\pi,\cdot)=\tfrac12(d\iota_{Z_0}d\tau)(Z_0,\cdot)-\tfrac122\beta((Z_0')^\nabla,\cdot)+h_0'''g(Jv,\cdot^\nabla).
\]
Let us define a vertical vector field $W$ such that
\[
g(W^\nabla,\cdot^\nabla)=\tfrac12(d\iota_{Z_0}d\tau)(Z_0,\cdot),
\]
then the vertical part of \((\infty''')\) can be written as
\[
g(\tfrac23 JT^*_{T^*_vv}v+J(Z'_0)^\nabla+JR^*_{Jv,v}Jv,\cdot)= g(W^\nabla,\cdot)-g(J(Z_0')^\nabla,\cdot)+h_0'''g(Jv,\cdot).
\]
Plugging in $(Jv)^v$, we deduce that
\[
g(\tfrac23 JT^*_{T^*_vv}v+JR^*_{Jv,v}Jv,Jv)=g(W^\nabla,Jv)+h_0''',
\]
where we used that $(Z_0')^\nabla$ is orthogonal to $v$. The first piece is
\begin{equation}
\tfrac{2}{3}(T^*_{T^*_vv}v,v)=\tfrac23g(v,T_{T^*_vv}v)=-\tfrac23g(v,T_vT^*_vv)=-\tfrac23g(T^*_vv,T^*_vv)=-\tfrac23|T^*|^2(\mu(v)).    
\end{equation}
The second piece is
\[
g(R^*_{Jv,v}Jv,v)=g(Jv,R(Jv,v)v)=K(\mu(v)).
\]
The third piece is given by Lemma \ref{lem: Z0 identity} as
\[
g(W^\nabla,Jv)=\tfrac{1}{2}K(\mu(v))-\tfrac13|T^*|^2(\mu(v)).
\]
Plugging the three pieces in the formula, we get
\begin{equation}
h_0'''=\tfrac{1}{2}K(\mu(v))-\tfrac13|T^*|^2(\mu(v)).
\end{equation}
Restricting the vertical part of \((\infty''')\) to the orthogonal of $Jv$, we get
\[
P\Big(\tfrac23 JT^*_{T^*_vv}v+JR^*_{Jv,v}Jv-W^\nabla\Big)=-2J(Z_0')^\nabla,
\]
where $P$ is the orthogonal projection on the orthogonal to $Jv$ (and $v$). We obtain
\[
(Z_0')^\nabla=\tfrac12P\Big(\tfrac23 T^*_{T^*_vv}v+R^*_{Jv,v}Jv+JW^\nabla\Big).
\]
The horizontal part of \((\infty''')\) yields
\[
g(v,T((Z_0'')^\pi,\cdot))+g((Z_0'')^\nabla,\cdot)-\beta((Z_0''')^\pi,\cdot)=\tfrac12\tau''_0(\ \cdot^\pi).
\]
Define the horizontal vector field $V$ such that
\[
g(V^\pi,\cdot)=\tfrac12\tau''_0(\ \cdot^\pi).
\]
Then we obtain
\[
T^*_{(Z_0'')^\pi}v+(Z_0'')^\nabla-d \pi V=J (Z_0''')^\pi.
\]
This equation is solved if we put 
\[
(Z_0'')^\nabla=0,\qquad Z_0'''=(-JT^*_{(Z_0'')^h}v-J(Z_0'')^\nabla+J V^\pi)^h.
\]
All together we found as claimed:
\[
h_\varepsilon=\varepsilon+\varepsilon^3\left(\tfrac{1}{2}K(\mu(v))-\tfrac13|T^*|^2(\mu(v))\right)=\varepsilon+\varepsilon^3\frac{\hat K\circ\mu}{2}.
\]

\section{Theorem \ref{thmB}, Step 2: Analyzing the Drift}\label{sec:drift-analysis}
Note that the connection $\nabla^{g,J}$ on $TQ$ induces a connection on the principal $S^1$-bundle $\mu\colon SQ\to P_\C(TQ)$, and that $\tau$ is its connection $1$-form. In particular, $d\tau=\mu^*\delta$, where $\delta$ is the closed $2$-form on $P_\C(TQ)$ introduced in the introduction; see equation \eqref{eq:delta}. By assumption, $(SQ,\omega_\varepsilon)$ is Zoll, where $\omega_\varepsilon=\varepsilon d\lambda-\pi^*\beta$. That is, there exists a vector field spanning the kernel of $\omega_\varepsilon$ and generating a Zoll flow. By the normal form theorem, Theorem \ref{thm: magnetic normal form}, proved in the previous section, it follows that also
\[
\left(SQ,\psi_\varepsilon^*\omega_\varepsilon=-\pi^*\beta+d\big((H_\varepsilon\circ\mu)\tau\big)+o(\varepsilon^4)\right)
\]
is Zoll. Denote by $X_{H_\varepsilon}^h$ the horizontal lift of the Hamiltonian vector field $X_{H_\varepsilon}$ of $H_\varepsilon$ with respect to the symplectic form $-\pi^*\beta+H_\varepsilon\delta$ on $P_\C(TQ)$. Further note that the vertical vector field $(Jv)^v$, is the unique vector field characterized by $\tau((Jv)^v)=1$ and $(Jv)^v\in \ker(d\mu)$.

\begin{lemma}\label{kernel of normal form}
Up to order $o(\varepsilon^4)$, the vector field $(Jv)^v+X_{H_\varepsilon}^h$ spans the kernel of $\psi_\varepsilon^*\omega_\varepsilon$.
\end{lemma}

\begin{proof}
We compute
\begin{equation*}
\begin{aligned}
\iota_{(Jv)^v+X_{H_\varepsilon}^h}\psi_\varepsilon^*\omega_\varepsilon &=\mu^*(-\pi^*\beta+H_\varepsilon\delta)(X_{H_\varepsilon}^h,\cdot)\\
&+d(H_\varepsilon\circ\mu)(X_{H_\varepsilon}^h)\tau(\cdot)+\tau((Jv)^v)\mu^*dH_\varepsilon(\cdot)+o(\varepsilon^4).
\end{aligned}
\end{equation*}
Using that $\tau((Jv)^v)=1$ and $d(H_\varepsilon\circ\mu)(X_{H_\varepsilon}^h)=0$, since $X_{H_\varepsilon}^h$ projects to the Hamiltonian vector field of $H_\varepsilon$, this reduces to
\[
\iota_{(Jv)^v+X_{H_\varepsilon}^h}\psi_\varepsilon^*\omega_\varepsilon=\mu^*(-\pi^*\beta+H_\varepsilon\delta)(X_{H_\varepsilon}^h,\cdot)+\mu^*dH_\varepsilon(\cdot)+o(\varepsilon^4).
\]
This vanishes up to order $o(\varepsilon^4)$, again because $X_{H_\varepsilon}^h$ projects to the Hamiltonian vector field of $H_\varepsilon$ with respect to the symplectic form $-\pi^*\beta+H_\varepsilon\delta$.
\end{proof}

We now describe the leading-order horizontal and vertical components of $X_{H_\varepsilon}$ with respect to the splitting of $T(P_\mathbb C(TQ))$ induced by $\nabla^{g,J}$. We claim that
\begin{equation}
 X_{H_\varepsilon}^\pi=\varepsilon^4Y+o(\varepsilon^4),\qquad X_{H_\varepsilon}^\nabla=\varepsilon^2Z+o(\varepsilon^2),
\end{equation}
where $Y$ and $Z$ are determined by
\begin{equation}
\beta(Y,\cdot)=-\frac{1}{8}d^h\hat K,\qquad \bar\beta(Z,\cdot)=\frac{1}{8}d^v\hat K.
\end{equation}
Indeed, $X_{H_\varepsilon}$ is defined by the Hamiltonian equation
\[
-\pi^*\beta(X_{H_\varepsilon},\cdot)+H_\varepsilon\delta(X_{H_\varepsilon},\cdot)=-dH_\varepsilon=-\frac{\varepsilon^4}{8}d\hat K.
\]
We now evaluate this equation separately on vertical and horizontal distribution. Since $H_\varepsilon=\frac{\varepsilon^2}{2}+O(\varepsilon^4)$ and the leading-order part of $\delta$ is purely vertical and equal to $\bar\beta$ (see eq. \ref{eq:delta}), the vertical component gives
\[
\frac{\varepsilon^2}{2}\beta(X_{H_\varepsilon}^\nabla,\cdot)=-\frac{\varepsilon^4}{8}d^v\hat K+o(\varepsilon^4).
\]
This implies
\[
X_{H_\varepsilon}^\nabla=\varepsilon^2 Z+o(\varepsilon^2),
\]
where $\beta(Z,\cdot)=\frac{1}{8}d^v\hat K$. Similarly, evaluating on horizontal vectors, the term involving $\delta$ is of higher order, so we obtain
\[
-\beta( X_{H_\varepsilon}^\pi,\cdot)=-\frac{\varepsilon^4}{8}d^h\hat K+o(\varepsilon^4).
\]
Hence
\[
X_{H_\varepsilon}^\pi=\varepsilon^4Y+o(\varepsilon^4),
\]
where $\beta(Y,\cdot)=-\frac{1}{8}d^h\hat K$. This proves the claimed decomposition.

The projected dynamics therefore exhibits two distinct drifts: a vertical drift of order $\varepsilon^2$ along the fibers of $P_{\mathbb C}(TQ)\to Q$, and a horizontal drift of order $\varepsilon^4$ along the base $Q$. As in Section~\ref{sec:conformally-compatible}, the strategy is to rule out Zollness by showing that any non-trivial projected drift forces the period $T_\varepsilon$ to diverge, contradicting the existence of the finite limit period $T_*$ established in Step~1.

Assume first that $d^v\hat K(\xi_0)\neq 0$ for some $\xi_0\in P_{\mathbb C}(TQ)$. Then $Z(\xi_0)\neq 0$. Since
\[
X^\nabla_{H_\varepsilon}=\varepsilon^2 Z+o(\varepsilon^2),
\]
the same flow-box argument as in Section~\ref{sec:conformally-compatible} shows that the projected orbit through $\xi_0$ remains embedded for times of order $\varepsilon^{-2}$. In particular, there exists a constant $a>0$ such that
\[
T_\varepsilon\ge a\varepsilon^{-2}
\]
for all sufficiently small $\varepsilon$. As $\varepsilon\to 0$, this contradicts the existence of the finite limit period $T_*$. Hence $d^v\hat K=0$.

Once $d^v\hat K=0$, the leading projected drift is horizontal and equals $\varepsilon^4 Y$. If $d^h\hat K\neq 0$ at some point, the same argument applied in the horizontal direction yields a constant $a'>0$ such that
\[
T_\varepsilon\ge a'\varepsilon^{-4}
\]
for all sufficiently small $\varepsilon$, again contradicting the existence of $T_*$. Therefore $d^h\hat K=0$ as well.

We conclude that both $d^v\hat K$ and $d^h\hat K$ vanish, and hence $\hat K$ is constant. This proves \eqref{curvature torsion constant}. If $J$ is integrable, then $N^J=0$, so $K^{g,J}$ is constant. It follows that $(Q,g,J)$ is a complex space form. In this case, as recalled in the Introduction, the magnetic flow is Zoll for all sufficiently small energy values. This completes the proof.

\bibliography{refs}

@article{B24,
  title={On symplectic geometry of tangent bundles of {H}ermitian symmetric spaces},
  author={Bimmermann, Johanna},
  journal={arXiv preprint arXiv:2406.16440},
  year={2024}
}

@article{Bott,
 author = {Bott, Raoul},
 title = {Nondegenerate critical manifolds},
 fjournal = {Annals of Mathematics. Second Series},
 journal = {Ann. Math. (2)},
 issn = {0003-486X},
 volume = {60},
 pages = {248--261},
 year = {1954},
 language = {English},
 doi = {10.2307/1969631},
 zbMATH = {3092823},
 Zbl = {0058.09101}
}

@article{BH,
 author = {Banyaga, Augustin and Hurtubise, David E.},
 title = {A proof of the {Morse}-{Bott} lemma.},
 fjournal = {Expositiones Mathematicae},
 journal = {Expo. Math.},
 issn = {0723-0869},
 volume = {22},
 number = {4},
 pages = {365--373},
 year = {2004},
 language = {English},
 doi = {10.1016/S0723-0869(04)80014-8},
 zbMATH = {2222520},
 Zbl = {1078.57031}
}

@article{CMP,
 author = {Contreras, Gonzalo and Macarini, Leonardo and Paternain, Gabriel P.},
 title = {Periodic orbits for exact magnetic flows on surfaces},
 fjournal = {IMRN. International Mathematics Research Notices},
 journal = {Int. Math. Res. Not.},
 issn = {1073-7928},
 volume = {2004},
 number = {8},
 pages = {361--387},
 year = {2004},
 language = {English},
 doi = {10.1155/S1073792804205050},
 zbMATH = {2207647},
 Zbl = {1086.37032}
}

@article{Schneider,
 author = {Schneider, Matthias},
 title = {Closed magnetic geodesics on {{\(S^2\)}}},
 fjournal = {Journal of Differential Geometry},
 journal = {J. Differ. Geom.},
 issn = {0022-040X},
 volume = {87},
 number = {2},
 pages = {343--388},
 year = {2011},
 language = {English},
 doi = {10.4310/jdg/1304514976},
 zbMATH = {5917819},
 Zbl = {1232.53006}
}

@article{CZ,
 author = {Cheng, Da Rong and Zhou, Xin},
 title = {Existence of curves with constant geodesic curvature in a {Riemannian} 2-sphere},
 fjournal = {Transactions of the American Mathematical Society},
 journal = {Trans. Am. Math. Soc.},
 issn = {0002-9947},
 volume = {374},
 number = {12},
 pages = {9007--9028},
 year = {2021},
 language = {English},
 doi = {10.1090/tran/8510},
 zbMATH = {7618823},
 Zbl = {1512.58008}
}

@article{AB,
 author = {Asselle, Luca and Benedetti, Gabriele},
 title = {The {Lusternik}-{Fet} theorem for autonomous {Tonelli} {Hamiltonian} systems on twisted cotangent bundles},
 fjournal = {Journal of Topology and Analysis},
 journal = {J. Topol. Anal.},
 issn = {1793-5253},
 volume = {8},
 number = {3},
 pages = {545--570},
 year = {2016},
 language = {English},
 doi = {10.1142/S1793525316500205},
 zbMATH = {6603943},
 Zbl = {1345.37059}
}

@article{NT,
 author = {Novikov, Serge\u{\i} P. and Ta{\u{\i}}manov, Iskander A.},
 title = {Periodic extremals of many-valued or not-everywhere-positive functionals},
 fjournal = {Soviet Mathematics. Doklady},
 journal = {Sov. Math., Dokl.},
 issn = {0197-6788},
 volume = {29},
 pages = {18--20},
 year = {1984},
 language = {English},
 zbMATH = {3938075},
 Zbl = {0585.58009}
}

@article{Merry0,
 author = {Merry, Will J.},
 title = {Closed orbits of a charge in a weakly exact magnetic field},
 fjournal = {Pacific Journal of Mathematics},
 journal = {Pac. J. Math.},
 issn = {1945-5844},
 volume = {247},
 number = {1},
 pages = {189--212},
 year = {2010},
 language = {English},
 doi = {10.2140/pjm.2010.247.189},
 zbMATH = {5809035},
 Zbl = {1246.37082}
}

@article{Merry1,
 author = {Merry, Will J.},
 title = {On the {Rabinowitz} {Floer} homology of twisted cotangent bundles},
 fjournal = {Calculus of Variations and Partial Differential Equations},
 journal = {Calc. Var. Partial Differ. Equ.},
 issn = {0944-2669},
 volume = {42},
 number = {3-4},
 pages = {355--404},
 year = {2011},
 language = {English},
 doi = {10.1007/s00526-011-0391-1},
 zbMATH = {5968978},
 Zbl = {1239.53111}
}

@article{Merry2,
 author = {Merry, Will J.},
 title = {Correction to ``{Closed} orbits of a charge in a weakly exact magnetic field''},
 fjournal = {Pacific Journal of Mathematics},
 journal = {Pac. J. Math.},
 issn = {1945-5844},
 volume = {280},
 number = {1},
 pages = {255--256},
 year = {2016},
 language = {English},
 doi = {10.2140/pjm.2016.280.255},
 zbMATH = {6528217},
 Zbl = {1332.37046}
}

@article{MG,
 author = {Groman, Yoel and Merry, Will J.},
 title = {The symplectic cohomology of magnetic cotangent bundles},
 fjournal = {Commentarii Mathematici Helvetici},
 journal = {Comment. Math. Helv.},
 issn = {0010-2571},
 volume = {98},
 number = {2},
 pages = {365--424},
 year = {2023},
 language = {English},
 doi = {10.4171/CMH/555},
 zbMATH = {7755954},
 Zbl = {1529.53079}
}

@article{Schlenk,
 author = {Schlenk, Felix},
 title = {Applications of {Hofer}'s geometry to {Hamiltonian} dynamics},
 fjournal = {Commentarii Mathematici Helvetici},
 journal = {Comment. Math. Helv.},
 issn = {0010-2571},
 volume = {81},
 number = {1},
 pages = {105--121},
 year = {2006},
 language = {English},
 doi = {10.4171/CMH/45},
 zbMATH = {5033715},
 Zbl = {1094.37031}
}

@article{FF,
 author = {Frauenfelder, Urs and Schlenk, Felix},
 title = {Hamiltonian dynamics on convex symplectic manifolds},
 fjournal = {Israel Journal of Mathematics},
 journal = {Isr. J. Math.},
 issn = {0021-2172},
 volume = {159},
 pages = {1--56},
 year = {2007},
 language = {English},
 doi = {10.1007/s11856-007-0037-3},
 url = {opus.bibliothek.uni-augsburg.de/opus4/files/56443/0303282v1.pdf},
 zbMATH = {5186271},
 Zbl = {1126.53056}
}

@article{Gin87,
 author = {Ginzburg, Viktor L.},
 title = {New generalizations of {Poincar{\'e}}'s geometric theorem},
 fjournal = {Functional Analysis and its Applications},
 journal = {Funct. Anal. Appl.},
 issn = {0016-2663},
 volume = {21},
 number = {1-3},
 pages = {100--106},
 year = {1987},
 language = {English},
 doi = {10.1007/BF01078023},
 zbMATH = {4071972},
 Zbl = {0656.58027}
}

@article{GK2,
 author = {Ginzburg, Viktor L. and Kerman, Ely},
 title = {Periodic orbits of {Hamiltonian} flows near symplectic extrema.},
 fjournal = {Pacific Journal of Mathematics},
 journal = {Pac. J. Math.},
 issn = {1945-5844},
 volume = {206},
 number = {1},
 pages = {69--91},
 year = {2002},
 language = {English},
 doi = {10.2140/pjm.2002.206.69},
 zbMATH = {1868174},
 Zbl = {1055.37065}
}

@incollection{GK,
 author = {Ginzburg, Viktor L. and Kerman, Ely},
 title = {Periodic orbits in magnetic fields in dimensions greater than two},
 booktitle = {Geometry and topology in dynamics. AMS special session on topology in dynamics, Winston-Salem, NC, USA, October 9--10, 1998 and the AMS-AWM special session on geometry in dynamics, San Antonio, TX, USA, January 13--16, 1999},
 isbn = {0-8218-1958-5},
 pages = {113--121},
 year = {1999},
 publisher = {Providence, RI: American Mathematical Society},
 language = {English},
 zbMATH = {1408395},
 Zbl = {0948.37045}
}

@article{CGK,
 author = {Cieliebak, Kai and Ginzburg, Viktor L. and Kerman, Ely},
 title = {Symplectic homology and periodic orbits near symplectic submanifolds},
 fjournal = {Commentarii Mathematici Helvetici},
 journal = {Comment. Math. Helv.},
 issn = {0010-2571},
 volume = {79},
 number = {3},
 pages = {554--581},
 year = {2004},
 language = {English},
 doi = {10.1007/s00014-004-0814-0},
 zbMATH = {2113905},
 Zbl = {1073.53118}
}

@article{Arn,
 author = {Arnol'd, Vladimir I.},
 title = {Some remarks on flows of line elements and frames},
 fjournal = {Soviet Mathematics. Doklady},
 journal = {Sov. Math., Dokl.},
 issn = {0197-6788},
 volume = {2},
 pages = {562--564},
 year = {1961},
 language = {English},
 zbMATH = {3201217},
 Zbl = {0124.14601}
}

@book{McDuff_Salamon_jhol,
  author    = {McDuff, Dusa and Salamon, Dietmar},
  title     = {J-Holomorphic Curves and Symplectic Topology},
  edition   = {2},
  series    = {American Mathematical Society Colloquium Publications},
  volume    = {52},
  publisher = {American Mathematical Society},
  address   = {Providence, RI},
  year      = {2012},
  isbn      = {978-0-8218-8746-2},
}

@article{Dom,
 author = {Dombrowski, Peter},
 title = {On the geometry of the tangent bundle},
 fjournal = {Journal f{\"u}r die Reine und Angewandte Mathematik},
 journal = {J. Reine Angew. Math.},
 issn = {0075-4102},
 volume = {210},
 pages = {73--88},
 year = {1962},
 language = {English},
 doi = {10.1515/crll.1962.210.73},
 url = {https://eudml.org/doc/150533},
 zbMATH = {3171620},
 Zbl = {0105.16002}
}

@article{AB_normal_forms_surfaces,
  author  = {Asselle, Luca and Benedetti, Gabriele},
  title   = {Normal forms for strong magnetic systems on surfaces: trapping regions and rigidity of {Z}oll systems},
  journal = {Ergodic Theory Dynam. Systems},
  volume  = {42},
  year    = {2022},
  pages   = {1871--1897},
  doi     = {10.1017/etds.2021.11},
}

@article{NN57,
  author    = {Newlander, August and Nirenberg, Louis},
  title     = {Complex analytic coordinates in almost complex manifolds},
  journal   = {Annals of Mathematics. Second Series},
  volume    = {65},
  number    = {3},
  year      = {1957},
  pages     = {391--404},
  issn      = {0003-486X},
  doi       = {10.2307/1970050},
  mrnumber  = {0088770},
}

@book{KN69,
  author = {Kobayashi, Shoshichi and Nomizu, Katsumi},
  title = {Foundations of Differential Geometry. Vol. II},
  series = {Interscience Tracts in Pure and Applied Mathematics},
  volume = {15},
  year = {1969},
  publisher = {Interscience Publishers},
  address = {New York}
}

@incollection{Don00,
 author = {Donaldson, Simon K.},
 title = {The {Seiberg}-{Witten} equations and almost-{Hermitian} geometry},
 booktitle = {Global differential geometry: the mathematical legacy of Alfred Gray. Proceedings of the international congress on differential geometry held in memory of Professor Alfred Gray, Bilbao, Spain, September 18--23, 2000},
 isbn = {0-8218-2750-2},
 pages = {32--38},
 year = {2001},
 publisher = {Providence, RI: American Mathematical Society (AMS)},
 language = {English},
 zbMATH = {1749238},
 Zbl = {1004.53016}
}

@article{Marle,
 author = {Marle, Charles-Michel},
 title = {Mod{\`e}le d'action hamiltonienne d'un groupe de {Lie} sur une vari{\'e}t{\'e} symplectique. ({Model} of {Hamiltonian} action of a {Lie} group on a symplectic manifold)},
 fjournal = {Rendiconti del Seminario Matematico},
 journal = {Rend. Semin. Mat., Torino},
 issn = {0373-1243},
 volume = {43},
 pages = {227--251},
 year = {1985},
 language = {French},
 zbMATH = {3967352},
 Zbl = {0599.53032}
}

@book{Audin,
 author = {Audin, Mich{\`e}le},
 title = {Torus actions on symplectic manifolds},
 edition = {2nd revised ed.},
 fseries = {Progress in Mathematics},
 series = {Prog. Math.},
 issn = {0743-1643},
 volume = {93},
 isbn = {3-7643-2176-8},
 year = {2004},
 publisher = {Basel: Birkh{\"a}user},
 language = {English},
 zbMATH = {2118464},
 Zbl = {1062.57040}
}

@article {usher,
    AUTHOR = {Usher, Michael},
     TITLE = {Floer homology in disk bundles and symplectically twisted
              geodesic flows},
   JOURNAL = {J. Mod. Dyn.},
  FJOURNAL = {Journal of Modern Dynamics},
    VOLUME = {3},
      YEAR = {2009},
    NUMBER = {1},
     PAGES = {61--101},
      ISSN = {1930-5311,1930-532X},
   MRCLASS = {53D40 (53D25)},
  MRNUMBER = {2481333},
MRREVIEWER = {Tobias\ Ekholm},
       DOI = {10.3934/jmd.2009.3.61},
       URL = {https://doi.org/10.3934/jmd.2009.3.61},
}

@incollection{Gin96,
 author = {Ginzburg, Viktor L.},
 title = {On closed trajectories of a charge in a magnetic field. {An} application of symplectic geometry},
 booktitle = {Contact and symplectic geometry},
 isbn = {0-521-57086-7},
 pages = {131--148},
 year = {1996},
 publisher = {Cambridge: Cambridge University Press},
 language = {English},
 zbMATH = {956493},
 Zbl = {0873.58034}
}

@article{Ginzburg_gurel,
 author = {Ginzburg, Viktor L. and G{\"u}rel, Ba{\c{s}}ak Z.},
 title = {Periodic orbits of twisted geodesic flows and the {Weinstein}-{Moser} theorem},
 fjournal = {Commentarii Mathematici Helvetici},
 journal = {Comment. Math. Helv.},
 issn = {0010-2571},
 volume = {84},
 number = {4},
 pages = {865--907},
 year = {2009},
 language = {English},
 doi = {10.4171/CMH/184},
 url = {www.ems-ph.org/journals/show_pdf.php?issn=0010-2571&vol=84&iss=4&rank=7},
 zbMATH = {5609488},
 Zbl = {1184.37046}
}

@article {Weinstein_Moser_W,
    AUTHOR = {Weinstein, Alan},
     TITLE = {Normal modes for nonlinear {H}amiltonian systems},
   JOURNAL = {Invent. Math.},
  FJOURNAL = {Inventiones Mathematicae},
    VOLUME = {20},
      YEAR = {1973},
     PAGES = {47--57},
      ISSN = {0020-9910,1432-1297},
   MRCLASS = {34C35 (58F05)},
  MRNUMBER = {328222},
MRREVIEWER = {Clark\ Robinson},
       DOI = {10.1007/BF01405263},
       URL = {https://doi.org/10.1007/BF01405263},
}

@article {Weinstein_Moser_Mo,
    AUTHOR = {Moser, Jürgen},
     TITLE = {Periodic orbits near an equilibrium and a theorem by {A}lan
              {W}einstein},
   JOURNAL = {Comm. Pure Appl. Math.},
  FJOURNAL = {Communications on Pure and Applied Mathematics},
    VOLUME = {29},
      YEAR = {1976},
    NUMBER = {6},
     PAGES = {724--747},
      ISSN = {0010-3640,1097-0312},
   MRCLASS = {58F05 (34C30)},
  MRNUMBER = {426052},
MRREVIEWER = {Alan\ Weinstein},
       DOI = {10.1002/cpa.3160290613},
       URL = {https://doi.org/10.1002/cpa.3160290613},
}

@article {Bottkol,
    AUTHOR = {Bottkol, Matthew},
     TITLE = {Bifurcation of periodic orbits on manifolds and {H}amiltonian
              systems},
   JOURNAL = {J. Differential Equations},
  FJOURNAL = {Journal of Differential Equations},
    VOLUME = {37},
      YEAR = {1980},
    NUMBER = {1},
     PAGES = {12--22},
      ISSN = {0022-0396,1090-2732},
   MRCLASS = {58F22 (58F05 58F14 70F07)},
  MRNUMBER = {583335},
MRREVIEWER = {Mark\ Levi},
       DOI = {10.1016/0022-0396(80)90084-4},
       URL = {https://doi.org/10.1016/0022-0396(80)90084-4},
}

@article {Kerman,
    AUTHOR = {Kerman, Ely},
     TITLE = {Periodic orbits of {H}amiltonian flows near symplectic
              critical submanifolds},
   JOURNAL = {Internat. Math. Res. Notices},
  FJOURNAL = {International Mathematics Research Notices},
      YEAR = {1999},
    NUMBER = {17},
     PAGES = {953--969},
      ISSN = {1073-7928,1687-0247},
   MRCLASS = {37J45 (53D05 70H12)},
  MRNUMBER = {1717637},
MRREVIEWER = {Karl\ Friedrich\ Siburg},
       DOI = {10.1155/S1073792899000501},
       URL = {https://doi.org/10.1155/S1073792899000501},
}

@book {McDuff_Salamon_intro,
    AUTHOR = {McDuff, Dusa and Salamon, Dietmar},
     TITLE = {Introduction to symplectic topology},
    SERIES = {Oxford Mathematical Monographs},
   EDITION = {Second},
 PUBLISHER = {The Clarendon Press, Oxford University Press, New York},
      YEAR = {1998},
     PAGES = {x+486},
      ISBN = {0-19-850451-9},
   MRCLASS = {53D35 (53D40 57R17 57R57 57R58)},
  MRNUMBER = {1698616},
MRREVIEWER = {Hansj\"org\ Geiges},
}

@article {Weinstein_nbd_thm,
    AUTHOR = {Weinstein, Alan},
     TITLE = {Symplectic manifolds and their {L}agrangian submanifolds},
   JOURNAL = {Advances in Math.},
  FJOURNAL = {Advances in Mathematics},
    VOLUME = {6},
      YEAR = {1971},
     PAGES = {329--346},
      ISSN = {0001-8708},
   MRCLASS = {57.50},
  MRNUMBER = {286137},
MRREVIEWER = {D.\ G.\ Ebin},
       DOI = {10.1016/0001-8708(71)90020-X},
       URL = {https://doi.org/10.1016/0001-8708(71)90020-X},
}

@article {AB_21,
    AUTHOR = {Abbondandolo, Alberto and Benedetti, Gabriele},
     TITLE = {On the local systolic optimality of {Z}oll contact forms},
   JOURNAL = {Geom. Funct. Anal.},
  FJOURNAL = {Geometric and Functional Analysis},
    VOLUME = {33},
      YEAR = {2023},
    NUMBER = {2},
     PAGES = {299--363},
      ISSN = {1016-443X,1420-8970},
   MRCLASS = {53C60 (53C22 53D10)},
  MRNUMBER = {4578460},
MRREVIEWER = {Louis\ Merlin},
       DOI = {10.1007/s00039-023-00624-z},
       URL = {https://doi.org/10.1007/s00039-023-00624-z},
}

@misc{sanjay2025c2localsystolicoptimalityzoll,
      title={On the {$C^2$}-local systolic optimality of {Z}oll odd-symplectic forms}, 
      author={Samanyu Sanjay},
      year={2025},
      eprint={2512.01937},
      archivePrefix={arXiv},
      primaryClass={math.SG},
      url={https://arxiv.org/abs/2512.01937}, 
}

@book {Hofer_Zehnder_book,
    AUTHOR = {Hofer, Helmut and Zehnder, Eduard},
     TITLE = {Symplectic invariants and {H}amiltonian dynamics},
    SERIES = {Birkh\"auser Advanced Texts: Basler Lehrb\"ucher.
              [Birkh\"auser Advanced Texts: Basel Textbooks]},
 PUBLISHER = {Birkh\"auser Verlag, Basel},
      YEAR = {1994},
     PAGES = {xiv+341},
      ISBN = {3-7643-5066-0},
   MRCLASS = {58-02 (34C25 57R15 58E05 58F05 70H05)},
  MRNUMBER = {1306732},
MRREVIEWER = {Daniel\ M.\ Burns, Jr.},
       DOI = {10.1007/978-3-0348-8540-9},
       URL = {https://doi.org/10.1007/978-3-0348-8540-9},
}
\bibliographystyle{alpha}
\end{document}